\documentclass[pdflatex,sn-mathphys-num]{sn-jnl}

\usepackage{graphicx}%
\usepackage{multirow}%
\usepackage{amsmath,amssymb,amsfonts}%
\usepackage{amsthm}%
\usepackage{mathrsfs}%
\usepackage[title]{appendix}%
\usepackage{xcolor}%
\usepackage{textcomp}%
\usepackage{manyfoot}%
\usepackage{booktabs}%
\usepackage{algorithm}%
\usepackage{algorithmicx}%
\usepackage{algpseudocode}%
\usepackage{listings}%
\newcommand{\R}{\mathbb{R}}

\usepackage{subfig}
\usepackage{epstopdf}
\usepackage{soul}

\DeclareMathOperator\supp{supp}
\newtheorem{theorem}{Theorem}[section]

\newtheorem{proposition}{Proposition}[section]
\newtheorem{corollary}{Corollary}[section]
\newtheorem{lemma}{Lemma}[section]
\theoremstyle{plain}

\theoremstyle{definition}

\newtheorem{remark}{Remark}[section]

\begin{document}

\title[Article Title]{\centering HV Metric For Time-Domain Full Waveform Inversion\\
\textit{\small (dedicated to Bj\"orn Engquist on the occasion of his 80th birthday)}}


\author*[1]{\fnm{Matej} \sur{Neumann}}\email{mn528@cornell.edu}

\author[1]{\fnm{Yunan} \sur{Yang}}\email{yunan.yang@cornell.edu}

\affil*[1]{Department of Mathematics, \orgname{Cornell University,}, \orgaddress{\street{212 Garden Ave}, \city{Ithaca}, \postcode{14850}, \state{New York}, \country{USA}}}

\abstract{Full‐waveform inversion (FWI) is a powerful technique for reconstructing high‐resolution material parameters from seismic or ultrasound data. The conventional least‐squares ($L^{2}$) misfit suffers from pronounced non-convexity that leads to \emph{cycle skipping}. Optimal-transport misfits, such as the Wasserstein distance, alleviate this issue; however, their use requires artificially converting the wavefields into probability measures, a preprocessing step that can modify critical amplitude and phase information of time-dependent wave data. We propose the \emph{HV metric}, a transport‐based distance that acts naturally on signed signals, as an alternative metric for the $L^{2}$ and Wasserstein objectives in time-domain FWI. After reviewing the metric's definition and its relationship to optimal transport, we derive closed‐form expressions for the Fr\'echet derivative and Hessian of the map $f \mapsto d_{\text{HV}}^2(f,g)$, enabling efficient adjoint‐state implementations. A spectral analysis of the Hessian shows that, by tuning the hyperparameters $(\kappa,\lambda,\varepsilon)$, the HV misfit seamlessly interpolates between $L^{2}$, $H^{-1}$, and $H^{-2}$ norms, offering a tunable trade-off between the local point-wise matching and the global transport-based matching.   Synthetic experiments on the Marmousi and BP benchmark models demonstrate that the HV metric-based objective function yields faster convergence and superior tolerance to poor initial models compared to both $L^{2}$ and Wasserstein misfits. These results demonstrate the HV metric as a robust, geometry‐preserving alternative for large‐scale waveform inversion.
}

\keywords{full waveform inversion, HV metric, optimal transport, convexity}


\pacs[MSC Classification]{46E36, 49N45, 65K10, 86A15}

\maketitle
\section{Introduction}

Inverse problems aim to recover a model parameter $m$ such that the synthetic data produced by a forward operator $F$ matches the observed data $g$:
\begin{equation}\label{eq:fwd}
    F(m) = g.
\end{equation}
Due to the nonlinearity and lack of direct invertibility of $F$, solving~\eqref{eq:fwd} typically requires computational methods, which cast the problem as a variational optimization:
\begin{equation}\label{eq:mmin}
    m^* = \arg\min_{m} \, \mathcal{D}\bigl(F(m), g\bigr),
\end{equation}
where $\mathcal{D}$ is a user-defined misfit functional measuring the discrepancy between the synthetic and observed data. Among many choices, the squared $L^2$ norm is widely used for its simplicity and favorable analytic properties, leading to the classical least-squares or mean squared error (MSE) formulation.

In this paper, we focus on solving a specific inverse problem: time-domain full-waveform inversion (FWI). Originally introduced in~\cite{lailly1983seismic,tarantola1984inversion}, FWI is a computational imaging technique that has demonstrated significant effectiveness in retrieving high-resolution characterizations of material properties, such as wave speed or density, in geophysical (seismic) and medical (ultrasound) applications. Unlike traditional imaging methods, FWI leverages the entire recorded wavefield, incorporating both amplitude and phase information, to iteratively refine a model of the subsurface or medium. By minimizing the discrepancy between simulated and observed wave data, FWI reconstructs parameter distributions with remarkable detail.

Despite its versatility, FWI has long faced technical challenges stemming from its strong non-convexity when formulated with the standard $L^2$ misfit functional~\cite{virieux2009overview}. Non-convexity often leads gradient-based optimization algorithms to become trapped in spurious local minima, a phenomenon known as ``cycle skipping.'' To address this issue, researchers have explored alternative discrepancy measures that better capture the geometry of waveform data. Among these, the Wasserstein distance from optimal transport theory has emerged as a promising candidate for improving FWI formulations~\cite{santambrogio2015optimal,villani2021topics,EFWass,engquist2016optimal,yang2018analysis,engquist2020optimal}. By reformulating the objective function with a transport-based metric, these methods mitigate non-convexity, offering new, scalable algorithms for large-scale seismic applications~\cite{yang2018application}.

However, while optimal transport theory offers compelling advantages, its application to seismic inversion is complicated by fundamental differences in signal properties. Optimal transport operates on probability distributions—non-negative measures that integrate to one—while seismic signals are often signed-valued and oscillatory in nature. This mismatch introduces theoretical constraints that limit the full potential of optimal transport-based methods and necessitate modifications or extensions to the framework~\cite{qiu2017full,Survey2,engquist2020optimal}. Addressing this mismatch remains a key challenge in enabling transport-based techniques to fully realize their potential in seismic applications.

To overcome limitations in classical optimal transport-based approaches, the HV metric was introduced as a transport-based distance specifically designed to handle signed signals~\cite{han2024hv}. Conceptually, the HV metric combines horizontal deformations (captured by the Wasserstein metric) with vertical deformations (quantified by the $L^2$ norm), using a geometric framework first proposed by Miller et al.~\cite{miller2001group}. Building on this foundation, Han et al.~\cite{han2024hv} significantly advanced the HV metric by analyzing the regularity and stability of minimizing geodesics and developing an efficient algorithm based on fixed-point iterations. Validated through numerical experiments, the HV metric facilitates a PDE-constrained optimization approach that jointly penalizes spatial and amplitude mismatches while preserving waveform structure~\cite{han2024hv}. Furthermore, the HV metric has demonstrated its strength as a robust data misfit functional in frequency-domain FWI, successfully applied in both seismic inversion and ultrasound imaging~\cite{zeng2025robust}.

In this work, we focus on the use of the HV metric as an objective functional in time-domain FWI and further investigate its analytical properties, including the formulation of the Fr\'echet derivative and local convexity, both of which are essential for the practical application of the HV metric in inversion. Given the wave nature of the underlying physics and the HV metric's connection to the Wasserstein metric from optimal transport (see Table~\ref{tab:comparison}), we argue that the HV metric is better suited for waveform inversion tasks in the time domain compared to the frequency domain~\cite{zeng2025robust}.

The rest of the paper is organized as follows. Section~\ref{background} reviews the definition of the HV metric and provides a summary of the key elements of FWI. In Section~\ref{fgrad}, we derive closed-form expressions for the Fr\'echet derivative and Hessian of the map $f \mapsto d_{\text{HV}}^2(f, g)$, enabling efficient adjoint-state implementations for gradient-based FWI solvers. Section~\ref{hessprop} examines the second-order properties of the HV metric and demonstrates that, for different choices of the hyperparameters $\kappa$, $\lambda$, and $\varepsilon$, the HV misfit interpolates between the $L^2$, $H^{-1}$, and $H^{-2}$ norms. Section~\ref{numex} presents numerical experiments that validate the theoretical framework and demonstrate the robustness of the HV metric on synthetic FWI benchmarks. Finally, Section~\ref{conclusion} offers concluding remarks and outlines promising directions for future research.

\section{Background}\label{background}

This section summarizes the essential ingredients of full-waveform inversion and the HV metric.

\subsection{Full-Waveform Inversion}\label{sec:fwi}

Full-waveform inversion (FWI) is a nonlinear inverse technique for estimating subsurface parameters from recorded wavefields~\cite{lailly1983seismic,tarantola1984inversion}.  
In the numerical experiments of Section~\ref{numex}, we adopt the acoustic wave equation as the forward model
\begin{equation}\label{forwardwave}
    \begin{cases}
        m(x)\,\partial_{t}^{2}u(x,t)-\Delta u(x,t)=s(x,t),\\[2pt]
        u(x,0)=0,\quad\partial_{t}u(x,0)=0,
    \end{cases}
\end{equation}
where $u(x,t)$ is the pressure field, $s(x,t)$ is the source term, and $m(x)=1/c^{2}(x)$ with $c(x)$ the wave speed.  
The forward operator $F\colon m\mapsto u$ for this inverse problem is nonlinear, so we recast the inverse problem as a PDE-constrained optimization:
\begin{equation}\label{eq:fwi_obj}
    \min_{m}\,J(m),\quad \text{with}\quad 
    J(m)=\sum_{r} \mathcal{D}\left(F(m)(x_{r},t), g(x_{r},t) \right),
\end{equation}
where $\mathcal{D}(f,g)$ quantifies the discrepancy between signals $f$ and $g$ with $g$ denoting the measured data and the sum runs over receiver locations $x_{r}$.  For example, if the $L^2$ norm is used, $J(m)=\frac12\sum_{r}\int_{0}^{T_{\max}}\! |F(m)(x_{r},t)-g(x_{r},t)|^{2}\,dt$.

Because $m(x)$ is discretized on a fine mesh, finite-difference gradients are impractical.  
Instead we employ the adjoint-state method~\cite{chavent2010nonlinear,plessix2006review}.  
Solving \eqref{forwardwave} yields the forward wavefield $u$; the adjoint field $w$ is obtained from
\begin{equation}\label{adjointwave}
    \begin{cases}
        m(x)\,\partial_{t}^{2}w(x,t)-\Delta w(x,t)=-R^{*}\frac{\delta \mathcal{D}}{\delta f},\\[2pt]
        w(x,T_{\max})=0,\quad\partial_{t}w(x,T_{\max})=0,
    \end{cases}
\end{equation}
where $R$ restricts a field to the receiver set and $R^{*}$ is its adjoint. The adjoint equation~\eqref{adjointwave} is solved backward in time, also referred to as  back-propagation.  The gradient of the misfit with respect to $m$ is then
\begin{equation}\label{gradform}
    \frac{\delta J(m)}{\delta m}(x)=-\int_{0}^{T_{\max}}\partial_{t}^{2}u(x,t)\,w(x,t)\,dt.
\end{equation}
The choice of misfit functional $\mathcal{D}$ affects only the forcing term  $-R^{*}\frac{\delta \mathcal{D}}{\delta f}$ in~\eqref{adjointwave}. Throughout this work, we update the parameter $m$ with the L-BFGS-B algorithm~\cite{LBFGS}, using Equation~\eqref{gradform} to evaluate the gradient direction.

\subsection{The HV Metric}
Miller and Younes introduced in~\cite{miller2001group} a Riemannian metric on the space of square-integrable signals that simultaneously penalizes \emph{vertical} (amplitude) and \emph{horizontal} (transport) deformations.  This metric resolves the main shortcomings of the plain $L^{2}$ norm, which is blind to horizontal shifts, while avoiding the restrictions of the Wasserstein distance, which requires non-negative signals of equal total mass.  The framework was refined in~\cite{trouve2005local} and, more recently, in~\cite{han2024hv}, where the authors carefully studied the distance in one dimension, analyzed the associated geodesic equations, and developed an efficient numerical scheme for its evaluation.

Let $\rho_0,\rho_1$ be two $L^2(\R)$ functions. The square of the HV metric between $\rho_0$ and $\rho_1$, denoted as $d_{\text{HV}}^2(\rho_0,\rho_1),$ is defined as 
\begin{equation}
\label{action}
    d_{\text{HV}}^2(\rho_0,\rho_1)=\inf_{(f,v,z)\in \mathcal{A}(\rho_0,\rho_1)}\frac{1}{2}\int_0^1\int_0^1\kappa v^2+\lambda v_x^2+\varepsilon v_{xx}^2+z^2dxdt,
\end{equation}
 where $\kappa,\varepsilon>0,\lambda\geq0$ are modeling parameters and $\mathcal{A}(\rho_0,\rho_1)$ is the set of all admissible paths satisfying 
\begin{align*}
\mathcal{A}(\rho_0,\rho_1) := \big\{(f,v,z) & \mid f\in L^2((0,1),L^2(0,1))\cap C([0,1],H^{-1}(0,1)), \\
& v\in L^{2}((0,1),H^{2}(0,1)\cap H_{0}^{1}(0,1)),\ z\in L^2((0,1),L^2(0,1)), \\
& \boxed{ \partial_t f + \partial_x f\cdot v = z \text{ weakly}\,, f(\cdot,0)=\rho_0, f(\cdot,1)=\rho_1 } \big \}.
\end{align*}

By Proposition in \cite[Prop.~3.1]{han2024hv}, it follows that for any $\rho_0,\rho_1\in H^1(0,1)$, the solution to the system of Euler--Lagrange equations \begin{align}
\varepsilon v_{xxxx} -\lambda v_{xx} + \kappa v + z f_x & = 0 \text{ weakly on } (0,1)^2 \label{EL1}\\
 v = 0   \; \text{ and } \;  v_{xx} & = 0 \,\,\text{ on } \partial \Omega \times [0,1] \; \label{EL2} \\
z_t + (zv)_x & = 0  \text{ weakly on } (0,1)^2  \label{EL3}\\
 f_t + f_x  v - z & = 0 \text{ weakly on } (0,1)^2 \label{EL4}\\
 f(\cdot,0)  =\rho_0, \; \text{ and }  \; f(\cdot,1)&=\rho_1, \label{EL5} 
\end{align}
with the condition of $(f,v,z)$ belonging to $\mathcal A(\rho_0, \rho_1)$ determines the minimizing path $(f,v,z)$ to~\eqref{action}.
The minimizers can be computed using an iterative minimization scheme~\cite{han2024hv}. The method alternates between fixing $v$ while finding the optimal $(f,z)$, and fixing $f$ while
searching for the optimal $(v, z)$.

The first of the two sub-problems fixes $v$ and searches for the optimal pair $(f,z)$ by minimizing the convex objective functional under a linear constraint:
\[
\min_{f, z} \frac{1}{2} \int_0^1 \int_0^1 z^2 \, dx \, dt,\quad  \text{s.t.}\quad  (f, v, z) \in \mathcal{A}(\rho_0,\rho_1).
\]
The first-order optimality conditions are given by the Euler--Lagrange equations \eqref{EL3}, \eqref{EL4}, and \eqref{EL5}:
\begin{align*}
    &z_t + (zv)_x = 0,\quad f_t + f_x v - z = 0,\quad f(\cdot, 0) = \rho_0, \, f(\cdot, 1) = \rho_1.
\end{align*}
The second sub-problem fixes the function $f$ and computes the optimal $(v,z)$ by minimizing the action \eqref{action} under constraint $\mathcal{A}(\rho_0,\rho_1)$. As before, this reduces to a quadratic optimization problem under linear constraints:
\[
\min_{v,z} \frac{1}{2} \int_0^1 \int_0^1 \kappa v^2 + \lambda v_x^2 + \varepsilon v_{xx}^2 + z^2 \, \mathrm{d}x \mathrm{d}t, \quad\text{s.t. } (f,v,z) \in \mathcal{A}(\rho_0,\rho_1).
\]
Using \eqref{EL1} and the constraint \( z = f_t + vf_x \) yields the following fourth-order boundary value problem for \( v \):
\[
   \begin{cases}
\varepsilon v_{xxxx} - \lambda v_{xx} + (\kappa + |f_x|^2)v = -f_t f_x \quad &\text{on } (0, 1)^2,\\
v = 0 \quad &\text{on } \{0, 1\} \times (0, 1),\\
v_{xx} = 0 \quad & \text{on } \{0, 1\} \times (0, 1).
   \end{cases}
\]
We repeat the minimization scheme for a predetermined number of iterations $N$ or until we reach a predefined tolerance $\delta>0$ to obtain the optimal $(f,v,z)$ tuple for~\eqref{action}.

We remark that the HV metric can also be viewed as an extension of optimal transport metrics. Table~\ref{tab:comparison} provides a comparison between the HV metric~\cite{han2024hv}, the dynamic formulation of the 2-Wasserstein ($W_2$) metric from optimal transport (OT)~\cite{villani2021topics}, and unbalanced optimal transport (UOT)~\cite{chizat2018scaling}. All these metrics can be interpreted as the optimal cost computed over a set of admissible paths; however, their distinctions lie in (i) the types of paths considered and (ii) the specific cost functionals used. 

In particular, the HV metric employs paths that satisfy the \emph{adjoint} of the continuity equation $\partial_t f = -\nabla\!\cdot (f\,v)$, augmented by a source term $z$.  In classical OT, by contrast, admissible paths obey the continuity equation itself.  This adjoint formulation $\partial_t f = -\nabla f \,\cdot \,v + z$ permits local creation and annihilation of mass and therefore enables the HV metric to compare \emph{signed} signals, whereas the standard OT and UOT setting is limited to non-negative functions.

\begin{table}[ht]
\centering
\renewcommand{\arraystretch}{1.4}
\begin{tabular}{@{}l@{\hspace{1em}}p{0.38\linewidth}@{\hspace{1em}}p{0.46\linewidth}@{}}
\toprule
\textbf{Metric} & \textbf{Admissible Path} & \textbf{Action functional} \\ \midrule
HV &
$\partial_t f = -\nabla f \cdot \,v + z$ &
$\displaystyle
\int_{0}^{1}
        \int_{0}^{1}
          \left(\kappa v^{2} + \lambda v_x^{2} + \varepsilon v_{xx}^{2} + z^{2}\right)\,dx
     \,dt$ \\[10pt]

OT &
$\partial_t f = -\nabla\!\cdot (f\,v)$ &
$\displaystyle
 \int_{0}^{1}\!\int_{0}^{1} v^{2}\,f\,dx\,dt$ \\[10pt]

UOT &
$\partial_t f = -\nabla\!\cdot (f\,v) + z\,f$ &
$\displaystyle
\int_{0}^{1}\!\int_{0}^{1} \bigl(v^{2}+z^{2}\bigr)\,f\,dx\,dt$ \\ \bottomrule
\end{tabular}
\caption{Dynamic formulations of the HV metric, quadratic-cost optimal transport (OT), and unbalanced optimal transport (UOT). While they share a similar variational structure, they differ in the types of paths considered and the specific cost functionals employed.}
\label{tab:comparison}
\end{table}

\section{The Fr\'echet Derivative of the HV Metric}\label{fgrad}
To employ the HV metric as the misfit in FWI, i.e., $\mathcal{D}(f,g) =d_{\text{HV}}^2(f,g)$ in~\eqref{eq:fwi_obj}, we require its Fr\'echet derivative with respect to the simulated data $f=F(m)$:
\[
    \frac{\partial d_{\text{HV}}^2(f,g)}{\partial f}\,.
\]
Once this expression is available, it enters the adjoint equation~\eqref{adjointwave} as the back-propagated source term, yielding the adjoint wavefield $w$.  Substituting $w$ into the gradient formula~\eqref{gradform} then provides the update direction for the model parameter $m$ in any gradient-based optimization scheme.

In the following, we present two different approaches to computing $\frac{\partial d_{\text{HV}}^2(f,g)}{\partial f}$. The first derivation is more intuitive and geometrical, while the other one uses perturbation theory.

\begin{lemma}
\label{lemma-grad}
For fixed functions $\rho_0,\rho_1\in H^1(0,1)$, define for every $s$ the parameterized function $g(s,x)=g_s(x)\in H^{1}(0,1)$ satisfying $g_0=\rho_0.$ If the tuple $(f,v,z)$  denotes the minimizing geodesic between $\rho_0$ and $\rho_1,$ then the derivative of $L(s) := d_{\text{HV}}^2(g_s,\rho_1)$ with respect to $s$ at $s = 0$ can be expressed as 
$$
L'(0) = \lim_{s\to 0}\frac{d_{\text{HV}}^2(g_s,\rho_1)-d_{\text{HV}}^2(g_0,\rho_1)}{s}=-\int_0^1 z(x,0)\, \partial_s g_s|_{s=0}(x) d x.
$$
Based on the arbitrariness of $g_s$, we further obtain the  Fr\'echet derivative of the functional $\Phi: \rho_0\mapsto d_{\text{HV}}^2(\rho_0, \rho_1)$ given by 
\begin{equation}
  \frac{\partial \Phi}{\partial \rho_0}  =   \frac{\partial d_{\text{HV}}^2(\rho_0,\rho_1)}{\partial \rho_0} = - z(x,0)\,.
\end{equation}
\end{lemma}

\begin{proof}
Let $(w,h)$ denote the HV geometry tangent vectors in the tangent space at point $\rho_0$ corresponding to the Eulerian perturbation $\partial_s g_s|_{s=0}$. That is, 
\begin{equation}\label{eq:adm_tangent}
    - w\partial_x g_0 +h = \partial_s g_s|_{s=0},
\end{equation} 
and
\[
\begin{cases}
\varepsilon w_{xxxx} -\lambda w_{xx} + \kappa w   = -h\, \partial_x g_0  \quad & \text{on } (0,1), \\
w  = 0  \quad & \text{at } x = 0,1, \\
w_{xx} = 0  \quad & \text{at } x = 0,1.
\end{cases}
\]
With the special inner product for the tangent space of the HV geometry~\cite[Sec.~2.2]{han2024hv}, we can express the directional derivative evaluated at the Eulerian perturbation direction $\partial_s g_s|_{s=0}(x) $ as 
\begin{align*}
   & \int_0^1 \frac{\partial J}{\partial \rho_0}(x) \,\partial_s g_s|_{s=0}(x) d x\\
    =& \lim_{s\to 0}\frac{d_{\text{HV}}^2(g_s,\rho_1)-d_{\text{HV}}^2(g_0,\rho_1)}{s}\\
   =&
   -\int_0^1 \Big(\kappa v(x,0)w(x,0)+\lambda v_x(x,0)w_x(x,0) \\
   &\quad +\varepsilon v(x,0)_{xx}w(x,0)_{xx}+z(x,0)h(x,0) \Big) d x.
\end{align*} 
Proposition~14 in \cite{han2024hv} implies that if $(\bar w,\bar h)$ is another pair of functions satisfying~\eqref{eq:adm_tangent}, then we have
\begin{align*}
  & \int_0^1 \frac{\partial J}{\partial \rho_0} (x)\partial_s g_s|_{s=0}(x)  d x \\
  =& -\int_0^1\Big( \kappa v(x,0)\bar w(x,0)+\lambda v_x(x,0)\bar w_x(x,0) \\
   &\quad +\varepsilon v(x,0)_{xx}\bar w(x,0)_{xx}+z(x,0)\bar h(x,0) \Big)d x.
\end{align*} 
We complete the proof by choosing $\bar w\equiv0$ and $\bar h=\partial_s g_s|_{s=0}$. 
\end{proof}

The proof using perturbation analysis describes an alternate way of thinking about the derivative $\frac{\partial d_{\text{HV}}^2(\rho_0,\rho_1)}{\partial \rho_0}$. We will expand on this idea further when discussing higher-order derivatives. By perturbing $\rho_0$ by $\rho_0 + s \theta(x)$ and fixing $\rho_1$, the solution to the system of Euler--Lagrange equations described in~\cite[Sec.~3.1]{han2024hv} will also become $(\rho^s,v^s, z^s)$, rather than $(f,v,z)$. That is,
\begin{align*}
\varepsilon {v^s}_{xxxx} -\lambda  {v^s}_{xx} + \kappa  {v^s} +  {z^s}  {\rho^s}_x & = 0 \text{ weakly on } (0,1)^2 \\
 v^s = 0   \; \text{ and } \;  v^s_{xx} & = 0 \,\,\text{ on } \partial \Omega \times [0,1] \;  \\
(z^s)_t + (z^sv^s)_x & = 0  \text{ weakly on } (0,1)^2  \\
 (\rho^s)_t + (\rho^s)_x  v^s - z^s & = 0 \text{ weakly on } (0,1)^2 \\
 \rho^s(\cdot,0)  =\rho_0 + s \theta(x) \; \text{ and }  \; \rho^s(\cdot,1)&=\rho_1.
\end{align*}
We make the following ansatz 
\begin{eqnarray*}
  \rho^s &=& f + s f_1 + s^2 f_2 + o(s^2),\\
  v^s &=& v + s v_1 + s^2 v_2 + o(s^2), \\
  z^s &=& z + s z_1 + s^2 z_2 + o(s^2). 
\end{eqnarray*}

\begin{proof}[Alternate Proof of Lemma \ref{lemma-grad}]
  Substituting $v^s$ and $z^s$ into the definition of the HV metric, we get 
\begin{align*}
    &d_\text{HV}^2(\rho_0 + s\theta,\rho_1)=\frac{1}{2}\int_0^1\int_0^1 \kappa (v + s v_1 + o(s))^2+\lambda (v + s v_1 + o(s))_x^2\\ &\hspace{4cm}+\varepsilon (v + s v_1 + o(s))_{xx}^2+(z + s z_1 +o(s) )^2 dxdt\,.
\end{align*}
Expanding and grouping terms by the exponents of $s$, we get 
$$
d_\text{HV}^2(\rho_0+s\theta,\rho_1)=d_\text{HV}^2(\rho_0,\rho_1)+s\int_0^1\int_0^1 \left( vv_1+v_x(v_1)_x+v_{xx}(v_1)_{xx}+zz_1 \right) dxdt+o(s).
$$
We can simplify the integral by using integration by parts and the transport condition $z_1=\frac{d}{dt}f_1+v{f_{x}}_1+v_1{f_{x}}$:
\begin{align}
     &\int_0^1\int_0^1 \left(v_1(\varepsilon {v_{xxxx}}-\lambda {v_{xx}}+\kappa v)+z\left(\frac{d}{dt}f_1+v{f_{x}}_1+v_1{f_{x}}\right)\right)dxdt \nonumber \\=
   &\int_0^1\int_0^1 \left(v_1(\varepsilon {v_{xxxx}}-\lambda {v_{xx}}+\kappa v+z{f_{x}})+z\left(\frac{d}{dt}f_1+v{f_{x}}_1\right)\right)dxdt \nonumber \\=
     &\int_0^1\int_0^1 \left(z\frac{d}{dt}f_1+zv{f_{x}}_1\right)dxdt \nonumber \\=&\int_0^1\int_0^1 z\frac{d}{dt}f_1dxdt+\int_0^1\int_0^1 zv{f_{x}}_1dxdt.\label{eq:1}
\end{align}
Doing integration by parts again on the second summand, we get 
\begin{align*}
     &\int_0^1\int_0^1 zv{f_{x}}_1dxdt =\int_0^1 (zvf_1)(1,t)dt -\int_0^1 (zvf_1)(0,t)dt -\int_0^1\int_0^1 \frac{d(zv)}{dx}fdxdt.
\end{align*}Conditions (\ref{EL2}) and (\ref{EL3}) make the boundary terms vanish, leaving us with the equality
$$
\int_0^1\int_0^1 zv{f_{x}}_1dxdt =\int_0^1\int_0^1 z_t f_1dxdt.
$$
Continuing the above computation for~\eqref{eq:1}, we end up with 
\begin{align*}
    &\int_0^1\int_0^1\left( z\frac{d}{dt}f_1+zv{f_{x}}_1\right)dxdt \\
    =&\int_0^1\int_0^1 \left(z\frac{d}{dt}f_1dxdt+f_1 \frac{d}{dt}z \right)dxdt\\
    =&\int_0^1\int_0^1 \frac{d}{dt}(zf_1) dx dt \\
    =&\int_0^1 z(x,1)f_1(x,1)dx -\int_0^1z(x,0)f_1(x,0)dx\,.
\end{align*}
Since $\rho^s(\cdot,0)  =\rho_0 + s \theta \; \text{and}  \; \rho^s(\cdot,1)=\rho_1$, we see that $$f_1(x,0)=\theta\; \text{and}  \;f_1(x,1)=0.$$
The derivative is thus equal to
\begin{align*}
&\lim_{s\rightarrow 0 }\frac{d_{\text{HV}}^2(\rho_0+s\theta,\rho_1)-d_{\text{HV}}^2(\rho_0,\rho_1)}{s}=\lim_{s\rightarrow 0 } \int_0^1\int_0^1\left( vv_1+v_x(v_1)_x+v_{xx}(v_1)_{xx}+zz_1\right)dxdt\\
&=\lim_{s\rightarrow 0 } \int_0^1 z(x,1)\left(f_1(x,1)-z(x,0)f_1(x,0)\right)dx=-\int_0^1z(x,0)\theta(x)dx.
\end{align*}
  \end{proof}

In Section~\ref{hessprop}, we will explore the (local) convexity properties of the squared HV metric and an important object of study will be the Hessian of the functional 
$$
\Phi: \rho_0\mapsto d_{\text{HV}}^2(\rho_0,\rho_1).
$$ 
We can obtain an analytic formula for the Hessian using ideas similar to those in the above proof. 
  \begin{lemma}
  \label{lemma-hess}
      For fixed $\rho_0,\rho_1\in H^1(0,1)$, let $\rho_0+s\theta\in H^{1}(0,1)$ be the perturbation of the first signal. If the tuple $(f,v,z)$ denotes the minimizing geodesic between $\rho_0$ and $\rho_1$ under the HV metric, then the following equation holds
\begin{equation}\label{eq:2nd_var}
d_\text{HV}^2(\rho_0+s\theta,\rho_1)-d_\text{HV}^2(\rho_0,\rho_1)=-s\int_0^1 z(x,0)\partial_sg_s(x)dx-\frac{s^2}{2}\int_0^1z_1(x,0)\theta(x)dx+o(s^2)\,,
\end{equation}
where $z_1$ appears in the second proof for Lemma~\ref{lemma-grad}.
  \end{lemma}
  \begin{proof}
Expanding the difference between the perturbed and initial systems gives us 
\begin{align*}
   d_\text{HV}^2(\rho_0+s\theta,\rho_1)-d_\text{HV}^2(\rho_0,\rho_1)= \iint \big(\kappa (v + s v_1 + s^2v_2+o(s^2))^2+\lambda (v + s v_1 +s^2v_2 +o(s^2))_{x}^2& \\
     +\varepsilon (v + s v_1 + s^2v_2+o(s^2))_{xx}^2 + (z + s z_1 +s^2z_2+o(s^2) )^2 \big) dxdt&
\end{align*}
Using Lemma~\ref{lemma-grad} and combining the terms multiplied by $s^2$, we get
\begin{align}
    &d_\text{HV}^2(\rho_0+s\theta,\rho_1)-d_\text{HV}^2(\rho_0,\rho_1)=-s\int_0^1z(x,0)\theta(x)dx+\\
    & \frac{s^2}{2} \iint \left( 2{v_{xx}}_2{v_{xx}}+2{v_{x}}_2{v_{x}}+2v_2v+\left({v_{xx}}_1\right)^2+\left({v_{x}}_1\right)^2+\left(v_1\right)^2+2zz_2+z_1^2\right) dxdt+o(s^2).
\end{align}
We split the last integrand into two parts and treat each separately. Firstly, we can transfer all the derivatives with respect to $x$ to $v$ using integration by parts: 
\begin{align}
&2\int_0^1\int_0^1 \left( {v_{xx}}_2{v_{xx}}+{v_{x}}_2{v_{x}}+v_2v+zz_2 \right) dxdt \nonumber \\
=& 2\int_0^1\int_0^1 \left(v_2\left({v_{xxxx}}-{v_{xx}}+v\right)+zz_2 \right) dxdt. \label{eq:2}
\end{align}
Expanding the set of Euler--Lagrange equations (\ref{EL1})-(\ref{EL5}) in $s$, we get the following expression:
$$
z_2=\frac{d}{dt}f_2+v_2{f_{x}}+v_1{f_{x}}_1+v{f_{x}}_2.
$$ Substituting it into~\eqref{eq:2}, we arrive at a similar expression as the one in the proof of Lemma~\ref{lemma-grad}:
\begin{align*}
&2\int_0^1\int_0^1 \left(v_2\left({v_{xxxx}}-{v_{xx}}+v\right)+zz_2\right)dxdt \\
=&2\int_0^1\int_0^1 \left(v_2\left({v_{xxxx}}-{v_{xx}}+v\right)+z\left(\frac{d}{dt}f_2+v_2{f_{x}}+v_1{f_{x}}_1+v{f_{x}}_2\right)\right)dxdt\\
=&2\int_0^1\int_0^1\left( v_2\left({v_{xxxx}}-{v_{xx}}+v+z{f_{x}}\right)+z\left(\frac{d}{dt}f_2+v_1{f_{x}}_1+v{f_{x}}_2 \right)\right)dxdt\\
=&2\int_0^1\int_0^1\left( z\frac{d}{dt}f_2+zv{f_{x}}_2\right)dxdt+2\int_0^1\int_0^1 zv_1{f_{x}}_1dxdt.
\end{align*}
Similarly as before, the first summand is $0$ since 
$$
\int_0^1\int_0^1\left( z\frac{d}{dt}f_2+zv{f_{x}}_2\right)dxdt=\int_0^1z(x,1)f_2(x,1)dx-\int_0^1z(x,0)f_2(x,0)dx
$$ and the fact that $f_2(x,1)=f_2(x,0) = 0$ due to  Equation~(\ref{EL5}).  

So far we have managed to simplify the expression and obtain
\begin{align*}
    &d_\text{HV}^2(\rho_0+s\theta,\rho_1)-d_\text{HV}^2(\rho_0,\rho_1)=-s\int_0^1z(x,0)\theta(x)dx+\\
    & \frac{s^2}{2} \iint \left(\left({v_{xx}}_1\right)^2+\left({v_{x}}_1\right)^2+\left(v_1\right)^2+z_1^2+2z{f_{x}}_1v_1\right) dxdt+o(s^2).
\end{align*}
Using integration by parts on $v_1$ and recalling the Euler--Lagrange expansions of $z_1=\frac{d}{dt}f_1+v{f_{x}}_1+v_1{f_{x}}$ and ${v_{xxxx}}_1-{v_{xx}}_1+v_1+z{f_{x}}_1+z_1{f_{x}}=0$ we get 

\begin{align*}
   &\int_0^1\int_0^1 \left( \left({v_{xx}}_1\right)^2+\left({v_{x}}_1\right)^2+\left(v_1\right)^2+z_1^2+2z{f_{x}}_1v_1 \right) dxdt\\
  = &\int_0^1\int_0^1\left( v_1\left({v_{xxxx}}_1-{v_{xx}}_1+v_1+z{f_{x}}_1\right)+z_1\left(\frac{d}{dt}f_1+v{f_{x}}_1+v_1{f_{x}}\right)+zv_1{f_{x}}_1\right)dxdt\\
=&\int_0^1\int_0^1 \left(v_1\left({v_{xxxx}}_1-{v_{xx}}_1+v_1+z{f_{x}}_1+z_1{f_{x}}\right)+z_1\frac{d}{dt}f_1+z_1v{f_{x}}_1+zv_1{f_{x}}_1\right)dxdt\\
=&\int_0^1\int_0^1\left( z_1\frac{d}{dt}f_1+{f_{x}}_1(z_1v+v_1z)\right)dxdt.
\end{align*}Another application of integration by parts allows us to transfer the derivative with respect to $x$ away from $f_1$. Continuing the above equation,
\begin{align*}
&\quad\int_0^1\int_0^1\left( z_1\frac{d}{dt}f_1+{f_{x}}_1(z_1v+v_1z)\right)dxdt \\
& = \int_0^1\int_0^1 \left( z_1\frac{d}{dt}f_1-f_1\frac{d}{dx}(z_1v+v_1z)\right) dxdt+\int_0^1\left( f_1(z_1v+v_1z)\big|^1_0 \right) dt\\
&=\int_0^1\int_0^1 \left(z_1\frac{d}{dt}f_1-f_1\frac{d}{dx}(z_1v+v_1z)\right) dxdt,
\end{align*} where the boundary terms all equal to $0$. Based on Equation (\ref{EL3}), we have
$$
\frac{d}{dt}z_1=-\frac{d}{dx}(v_1z+z_1v)\,.
$$
We can then simplify the expression further, 
\begin{align*}
&\quad \int_0^1\int_0^1 \left(z_1\frac{d}{dt}f_1-f_1\frac{d}{dx}(z_1v+v_1z)\right)dxdt\\
&=\int_0^1\int_0^1 \left(z_1\frac{d}{dt}f_1+f_1\frac{d}{dt}z_1\right)dxdt\\
&=\int_0^1\int_0^1 \frac{d}{dt}\left(z_1f_1\right)dxdt\\
&=\int_0^1z_1(x,1)f_1(x,1)-z_1(x,0)f_1(x,0)dx=-\int_0^1z_1(x,0)f_1(x,0)dx\,.
\end{align*}
Finally, we obtain~\eqref{eq:2nd_var}.
\end{proof}

\section{Properties of the Hessian Operator}\label{hessprop}
In this section, we analyze the local convexity of the HV metric by studying the Hessian operator for the functional
\[
\Phi: \rho_0 \mapsto d_{\text{HV}}^2(\rho_0, \rho_1).
\] 
Our motivation for this analysis stems from the observation that the Hessian operator for the $L^2$ norm-based objective functional
\[
\rho_0 \mapsto \frac{1}{2}\|\rho_0 - \rho_1\|_2^2
\]
is the identity operator $I$, whereas the Hessian operator for the $2$-Wasserstein metric-motivated objective functional
\[
\rho_0 \mapsto W_2^2(\rho_0, \rho_1)
\]
is $\left( \nabla\cdot  \left(\rho_0 \nabla\right)\right)^{-1}$ when $\rho_0 = \rho_1$. In other words, the squared $W_2$ metric is asymptotically equivalent to the $\rho_0$-weighted $\dot{H}^{-1}$ semi-norm.

When comparing the spectral properties of the two operators, $I$ and $\left( \nabla\cdot  \left(\rho_0 \nabla\right)\right)^{-1}$, the latter places greater emphasis on the low-frequency components of the perturbation while being less sensitive to high-frequency perturbations. This distinction is one of the fundamental reasons why the $W_2$ metric, when used to quantify data misfits, is more robust to data noise and exhibits a wider basin of attraction when addressing translation- and dilation-type signal changes~\cite{engquist2020quadratic}.

Since the local second-order information of the objective functional can shed light on such important properties of the given metric under study, we will investigate the local behavior of the squared HV metric by analyzing its Hessian operator. This approach allows us to understand how the HV metric penalizes perturbations of different frequencies, especially in comparison to other commonly used metrics such as the $L^2$ norm and the $2$-Wasserstein metric.

\subsection{Deriving the Hessian Operator}
Similar to the linearization of the squared  $2$-Wasserstein metric first established in~\cite{otto2000generalization}, we consider the special case when $\rho_0=\rho_1=\rho$. In this setting, the tuple $(f,v,z)$ minimizing the HV distance is $(\rho,0,0)$ and the perturbation equations simplify to 
\begin{eqnarray*}
  \rho^s &=& \rho + s f_1 + s^2 f_2 + o(s^2)\\
  v^s &=& 0 + s v_1 + s^2 v_2 + o(s^2) \\
  z^s &=& 0 + s z_1 + s^2 z_2 + o(s^2)\\
 \rho^s(\cdot,0)  &=&\rho + s \theta(x) \; \text{ and }  \; \rho^s(\cdot,1)=\rho.
\end{eqnarray*}
Basd on Equation (\ref{EL3}), we note that $z_1(x,t)$ is time independent since 
\begin{align*}
    \frac{d}{dt}z_1=\left(v_1z+z_1v\right)_x=0\implies z_1(x,t)=z_1(x).
\end{align*}
As a consequence, $v_1(x,t)$ is also time-independent due to Equation~(\ref{EL1}), Equation~(\ref{EL3}) and $\rho$ all being independent of time $t$. This allows us to find an analytic expression for the variable $z_1$. 

Integrating~\eqref{EL4} with respect to $t$ and collecting the terms multiplied by $s$, we get 
\begin{align*}
    f_1(x,1)-f_2(x,0)=z_1-v_1\rho_x. 
\end{align*} 
Using the boundary conditions~\eqref{EL5}, we have 
\begin{align}
    z_1=v_1\rho_x -\theta.\label{z1}
\end{align}
Substituting~\eqref{z1} into (\ref{EL1}) allows us to express $v_1$ as the solution of a differential operator 
$$
L_{\rho}=\varepsilon\Delta^2-\lambda\Delta+I(\kappa+\rho_x^2),
$$  
with boundary conditions $v=0$ and $v_{xx} = 0$ at $x\in \{0,1\}$.
That is, $v_1=L_{\rho}^{-1}(\theta\rho_x)$ and consequently 
$
z_1=\rho_xL_{\rho}^{-1}(\theta\rho_x)-\theta$ based on~\eqref{z1}.

According to Lemma~\ref{lemma-hess},
\begin{eqnarray}
d_\text{HV}^2(\rho_0+s\theta,\rho_1)-d_\text{HV}^2(\rho_0,\rho_1) 
&=& \frac{s^2}{2} \int\theta(I-M_{\rho_x}L^{-1}_{\rho}M_{\rho_x})\theta dx +o(s^2)  \label{Hs}
\end{eqnarray}

where $M_f$ is the linear multiplication operator $M_f(g)=fg$.  Therefore, the local convexity of the squared HV metric, more precisely, the Hessian of the functional $\Phi: \rho_0 \mapsto d_\text{HV}^2(\rho_0, \rho_1)$ at $\rho_0 = \rho_1$ is determined by the following operator
\begin{equation}\label{eq:Binverse}
I - M_{\rho_x}L^{-1}M_{\rho_x}=I - \rho_x(\varepsilon\Delta^2-\lambda\Delta+I(\kappa+\rho_x^2))^{-1}\rho_x,
\end{equation}
which we will refer to as the \textit{Hessian operator} hereafter. We remark that our analysis is based on the assumption that $\rho_0 = \rho_1 = \rho$. One would expect to get a different asymptotic expansion~\eqref{Hs} if linearizing around $\rho_0\neq \rho_1$.
Moreover, we can split the integral \eqref{Hs} into the sum of two integrals 
$$
\int_{A_{1}}\theta^2(x)dx+\int_{A_{2}}\theta(I-M_{\rho_x}L^{-1}_{\rho}M_{\rho_x})\theta dx
$$ where $A_{1}:=\{x| \rho_{x}(x)=0\},A_{2}:=\{x| \rho_{x}(x)\neq0\}$ and $A_1\cup A_2=[0,1].$ The first integral represents the contribution of the $L^2$ norm to the Hessian, while the second integral represents the contribution. With this separation in mind, we will mostly focus on the case when the measure of the set $A_1$ is equal to 0.
\subsection{Properties of the Operator $B$}\label{subsec:B_def}

To perform a detailed analysis of the Hessian operator given in~\eqref{eq:Binverse}, 
we introduce the following linear operator $B$, defined by
\begin{equation}\label{eq:B}
        Bv=\frac{1}{\rho_x}\left(\varepsilon\left(\frac{v}{\rho_x}\right)_{xxxx}-\lambda\left(\frac{v}{\rho_x}\right)_{xx}+\frac{v}{\rho_x}\left(\kappa+\rho_x^2\right)\right),
\end{equation}
which is closely related to~\eqref{eq:Binverse}. The operator $B$ has a form similar to $\Delta^2-\Delta+I$, suggesting that a suitable domain $D(B)$ would be

$$
D(B)=\left\{w\bigg|\frac{w}{\rho_x}\in H^2\right\}.
$$ Before we determine the range $R(B),$ let us discuss some of the operator's properties.
\begin{proposition}
\label{inner}
Consider $B$ given in~\eqref{eq:B} with $v$ satisfying the following boundary condition
\[
\begin{cases}
            v(0)=v(1)=0\\
            \left(\frac{v}{\rho_x}\right)_{xx}(0)=\left(\frac{v}{\rho_x}\right)_{xx}(1)=0.
\end{cases}
\]
Then $B$ is self-adjoint and positive definite when viewed as an operator  $B:D(B)\subset L^2(0,1)\to L^2(0,1)$.
\end{proposition}

\begin{proof}
    We first show that the operator $B$ is self-adjoint. It holds that
    \begin{align*}
        \langle B v,w\rangle=\bigg \langle \frac{1}{\rho_x}\left(\varepsilon\left(\frac{v}{\rho_x}\right)_{xxxx}-\lambda\left(\frac{v}{\rho_x}\right)_{xx}+\frac{v}{\rho_x}\left(\kappa+\rho_x^2\right)\right),w\bigg\rangle\\
        =\bigg\langle \varepsilon\left(\frac{v}{\rho_x}\right)_{xxxx}-\lambda\left(\frac{v}{\rho_x}\right)_{xx}+\frac{v}{\rho_x}\left(\kappa+\rho_x^2\right),\frac{w}{\rho_x}\bigg\rangle.
    \end{align*}
 Performing integration by parts and using the boundary conditions, we get
 $$
 \bigg\langle{v},\frac{1}{\rho_x}\left(\varepsilon\left(\frac{v}{\rho_x}\right)_{xxxx}-\lambda\left(\frac{v}{\rho_x}\right)_{xx}+\frac{w}{\rho_x}\left(\kappa+\rho_x^2\right)\right)\bigg\rangle=\langle v,Bw\rangle.
 $$ Positive definiteness is shown in a similar way using integration by parts and taking into account the boundary conditions
 $$\langle Bv,v\rangle=\int_0^1 \varepsilon\left( \left(\left(\frac{v}{\rho_x}\right)_{xx}\right)^2+\lambda\left(\left(\frac{v}{\rho_x}\right)_{x}\right)^2+\frac{v^2}{\rho_x^2}(\kappa+\rho_x^2) \right) dx \geq0.$$
 The equality holds if and only if $v=0$ making $B$ strictly positive definite.
\end{proof}
\begin{remark}
    Since $B$ is strictly positive definite, we know that $B$ is injective. In particular, we can define $B^{-1}: R(B)\to D(B).$ It is a simple check that
    \[
    B^{-1} = M_{\rho_x}L^{-1}M_{\rho_x}=\rho_x(\varepsilon\Delta^2-\lambda\Delta+I(\kappa+\rho_x^2))^{-1}\rho_x
    \]
    following~\eqref{eq:Binverse}.
\end{remark}Next, we define a new inner product on the real vector space of measurable functions given by
$$
\langle f,g\rangle_B:=\int f\, Bg\, dx.
$$ 
Proposition \ref{inner} below asserts that this is a valid inner product and gives an expression for the norm
\begin{equation}
\label{bnorm}
\|v\|^2_B=\int_0^1\left(\varepsilon\left(\left(\frac{v}{\rho_x}\right)_{xx}\right)^2+\lambda\left(\left(\frac{v}{\rho_x}\right)_{x}\right)^2+\frac{v^2}{\rho_x^2}(\kappa+\rho_x^2)\right)dx\geq0.
\end{equation}
Next we assert that a function $v$ belongs to $D(B) = \{w: w\rho_x^{-1}\in H^2\}$ if and only if $\|v\|_B<\infty$.

\begin{proposition}
Let $|\rho_x|\in[0,M]$ and assume that $\kappa,\epsilon>0$ and $\lambda\geq0.$ Then $\frac{v}{\rho_x}\in H^2$ if and only if $\|v\|_B<\infty.$
\end{proposition}
\begin{proof} We first look at the case when $\lambda>0.$
Expanding the LHS of the equivalence, we get 
    \begin{align}
    \label{expanded}
        \left\|\frac{v}{\rho_x}\right\|^2_{H^2}=\int\left(\frac{v}{\rho_x}\right)^2dx+\int \left(\left(\frac{v}{\rho_x}\right)_x\right)^2dx+\int \left(\left(\frac{v}{\rho_x}\right)_{xx}\right)^2dx.
    \end{align}
    Denoting by $A_i$ the i-th summand of the RHS and recalling the expression for $\|v\|_B$ from~\eqref{bnorm}, we see that $A_2$ and $A_3$ differ from their counterparts in $\|v\|_B$ only by a positive constant, whereas $A_1$ bounds the remaining term by considering 
    $$
    \kappa A_1\leq\int\frac{v^2}{\rho_x^2}(\kappa+\rho_x^2) \leq(\kappa+M^2)A_1.
    $$ 
    By letting  $C_1=\min\{\kappa,\varepsilon,\lambda\}$ and $C_2=\max\{\kappa+M^2,\varepsilon,\lambda\}$, we get the following chain of inequalities showing the equivalence:
    \begin{equation}
    \label{equiv}
        C_1\left\|\frac{v}{\rho_x}\right\|_{H^2}^2\leq\|v\|_B^2\leq C_2\left\|\frac{v}{\rho_x}\right\|_{H^2}^2.
    \end{equation}
    If $\lambda=0,$ then we use the Poincar\'e inequality for the middle term in \eqref{expanded} to get 
    $$
     \left\|\frac{v}{\rho_x}\right\|^2_{H^2}\leq\int\left(\frac{v}{\rho_x}\right)^2dx+\left(1+C\right)\int \left|\left(\frac{v}{\rho_x}\right)_{xx}\right|^2dx,
    $$ where C is the Poincar\'e constant. In this case, \eqref{equiv} holds with constants $C_1=\min\left\{\kappa,\frac{\varepsilon}{1+C}\right\}$ and $C_2=\max\{\kappa+M^2,\varepsilon\}$.
\end{proof}
\begin{proposition}
    $(D(B),\|\cdot\|_B)$ as defined above is a Hilbert space.  
\end{proposition}
\begin{proof}
   We already know that $(D(B),\|\cdot\|_B)$ is a normed space with an inner product, so we only need to argue completeness. Let $\{v_n\}_{n=1}^{\infty}\subset D(B)$ be a Cauchy sequence. In particular, for every $\varepsilon>0$, there exists $n_0$ such that for every $n,m\geq n_0$ the following inequality holds $\|v_n-v_m\|_B\leq\varepsilon.$ By the first inequality of (\ref{equiv}), it follows that $\{\frac{v_n}{\rho_x}\}$ is a Cauchy sequence in $H^2.$ Since $H^2$ is complete, there exists a limit $w=\lim_{n\to\infty}\frac{v_n}{\rho_x}.$ To get a candidate for the limit in $D(B)$, we rewrite $w$ as $w=\frac{v}{\rho_x}$ where $v: = w\rho_x\in D(B)$. Using the second inequality of~(\ref{equiv}), we conclude that $\|v-v_n\|\leq \varepsilon$ for every positive $\varepsilon$ and $n$ large enough.
\end{proof}

With the $D(B)$ defined, we can characterize the range $R(B)$. We start by finding the dual space $D(B)^*.$ Let $\theta\in D(B)^*$ be a bounded linear functional. Since $(D(B),\|\cdot\|_B)$ is a Hilbert space, we can use the Riesz representation theorem to find a unique $w\in D(B)$ such that 
\begin{equation}\label{eq:B_range}
\theta(v)=\langle v,w \rangle_B=\int v\, Bw\, dx =\int\frac{v}{\rho_x}\, (\rho_xBw).
\end{equation}
Without loss of generality, we assume that $\|v\|_B = 1$. Since the linear functional $\theta$ is a bounded operator acting on $D(B)$ and $\frac{v}{\rho_x} \in H^2$, we obtain the requirement that $\rho_x B w \in H^{-2}$, as the integral takes the form of the standard $(H^2, H^{-2})$ duality pairing. Consequently, the dual of $D(B)$ is given by $D(B)^* = \{w \mid w \rho_x \in H^{-2}\}$.

It is straightforward to verify that for any $v \in D(B)$, $Bv \in D(B)^*$. Thus, $R(B) \subseteq D(B)^*$. On the other hand, for any $\theta \in D(B)^*$, Equation~\eqref{eq:B_range} implies $
    \theta(v) = \int v \, Bw \, dx$ for any $v\in D(B)$, 
which leads to $\theta = Bw \in R(B)$ based on the arbitrariness of $v$. As a result, we also have $D(B)^* \subseteq R(B)$, and we conclude that $D(B)^* = R(B)$.

\begin{proposition}
\label{embed}
If $\rho_x\in[0,M]$ then 
\begin{enumerate}
\item[(i)] $D(B)=\{w|  \|w\|_B<\infty\}$ continuously embeds into $L^2(0,1)$, and
\item[(ii)] $L^2(0,1)$ continuously embeds into $D(B)^*=\{w|w\rho_x\in H^{-2}\}$.
\end{enumerate}
\end{proposition}

\begin{proof}
    We begin by proving $(i)$. From (\ref{bnorm}) we see that 
    $$
    \|w\|_B^2\geq\int_0^1v^2\left(1+\frac{\kappa}{\rho_x^2}\right)dx\geq\left(1+\frac{\kappa}{M^2}\right)\|w\|_{L^2}^2.
    $$
    The proof of $(ii)$ is done similarly. Recall 
    $$
    \|w\rho_x\|_{H^{-2}}=\|(1+|\cdot|^2)^{-1}\mathcal{F}(w\rho_x)\|_{L^2}\leq\|\mathcal{F}(w\rho_x)\|_{L^2}=\|w\rho_x\|_{L^2}\leq M\|w\|_{L^2}.
    $$
\end{proof}

Proposition \ref{embed} allows us to continuously embed domain $D(B)$ first into $L^2$ and then into the codomain $D(B)^*.$ This makes $B:D(B)\subset L^2\subset D(B)^*\to D(B)^*$ a well defined operator. The inner product in $D(B)^*$ is defined in the usual way using Reisz representation theorem. In particular, let $v,w\in D(B)$. Using Riesz representation,  we can identify $v \in D(B)$ with its dual $v_D$ and compute the inner product 
$$
\langle v_D, Bw\rangle_{D^*}=\langle v,w\rangle_D=\langle v,Bw \rangle_{D,D^*},
$$
where $\langle \cdot\,,\, \cdot\rangle_{D}$, $\langle\cdot\, ,\, \cdot\rangle_{D^*}$ and $\langle \cdot\, ,\, \cdot \rangle_{D,D^*}$ represent the inner product in $D(B)$, the inner product in $D(B)^*$ and the duality pairing between $D(B)$ and $D(B)^*$, respectively.

\subsection{Spectrum of the Operator $B$}\label{subsec:B_spec}

Similarly to the Laplacian, which can be viewed as a map from either $H^{1}$ to $H^{-1}$ or a map from $H^{2}$ to $L^2$, we will focus on the case where $|\rho_x|$ is bounded from above, $D(B)=\{w \big| \frac{w}{\rho_x}\in H^{4}\}$ and consequently $R(B)\subseteq L^2$. 

Defining the domain and range in this way allows us to study the spectrum of the operator $B$. While we do not provide a complete description of the spectrum of $B$, we can produce bounds which are sufficient to show (local) convexity of the squared HV metric. 

Recall
$$
Bv=\frac{1}{\rho_x}\left(\varepsilon\left(\frac{v}{\rho_x}\right)_{xxxx}-\lambda\left(\frac{v}{\rho_x}\right)_{xx}+\frac{v}{\rho_x}\left(\kappa+\rho_x^2\right)\right).
$$
\begin{proposition}
\label{numradius}
    Let $|\rho_x|\in[0,M].$ Then $\inf\sigma(B)>1+\frac{\kappa}{M^2}.$ In particular $\sigma(B^{-1})\subset[0,\frac{M^2}{M^2+\kappa}]$ and the operator $I-B^{-1}$ is positive definite.
\end{proposition}
\begin{proof}
 Let $v\in D(B)$. Since $B$ is essentially self-adjoint when thought of as an operator $B: D(B)\subset L^2\to L^2$, we can get a bound for the infimum of the spectrum,
    \begin{align*}
    &\inf\sigma(B)\\
    =&\inf_{v\in D,\|v\|_{L^2}=1}\langle v_,Bv\rangle_{D,D^*}
    \\=&\inf_{v\in D,\|v\|_{L^2}=1}\left(\varepsilon^2\left\|\left(\frac{v}{\rho_x}\right)_{xx}\right\|^2_{L^2}+\lambda^2\left\|\left(\frac{v}{\rho_x}\right)_x\right\|^2_{L^2}+\left\|v\sqrt{1+\frac{\kappa}{\rho_x^2}}\right\|^{2}_{L^2}\right)\\
    \geq& \left(1+\frac{\kappa}{M^2}\right)\inf_{v\in D,\|v\|_{L^2}=1}\|v\|^{2}_{L^2}\\
     \geq&1+\frac{\kappa}{M^2}.
  \end{align*}
   The spectrum of $B$ is contained in the interval $\bigg[1+\frac{\kappa}{M^2},\infty\bigg)$ and thus $$\sigma(B^{-1})\subset\bigg[0,\frac{1}{1+\frac{\kappa}{M^2}}\bigg]=\bigg[0,\frac{M^2}{M^2+\kappa}\bigg].$$
    To find bounds for the spectrum of $I-B^{-1}$, we use the fact that for commuting operators $A,B$ on a Banach space the spectrum of the sum is contained in the sum of the spectra:
    $$
    \sigma(I-B^{-1})\subset\sigma(I)-\sigma(B^{-1})=1-\bigg[0,\frac{M^2}{\kappa +M^2}\bigg]=\bigg[\frac{M^2}{M^2+\kappa},1\bigg].
    $$ Since $I-B^{-1}$ is self-adjoint and has a positive spectrum, it is positive definite.
\end{proof} 

We present a stronger result of Proposition \ref{numradius} under additional assumption on $\rho_x.$ 

\begin{proposition}
\label{loewner}
Let $\preceq$ represent the Loewner ordering of operators. For $\theta\in L^2(0,1)$, we have the following inequalities:
\begin{enumerate}
        \item[(i)] If $\|\rho_x\|_{H^2(0,1)}<\infty$ and $0<m\leq|\rho_x|\leq M,$ almost everywhere, then there exist constants $C_1,C_2>0$ such that 
        \begin{align}
           & \kappa \left\|(C_1M^2\varepsilon\Delta^2-C_1M^2\lambda \Delta+I(\kappa+M^2))^{-\frac{1}{2}}\theta \right\|^2_{L^{2}} \nonumber  \\ 
           \leq&\bigg\langle\hat{\theta},\left(\frac{C_1(\varepsilon\langle\xi\rangle^4+\lambda\langle\xi\rangle^2)+\frac{\kappa}{M^2}}{C_1(\varepsilon\langle\xi\rangle^4+\lambda\langle\xi\rangle^2)+1+\frac{\kappa}{M^2}}\right)\hat{\theta}\bigg\rangle_{L^2} \nonumber \\
           \leq&\langle\theta,(I-B^{-1})\theta\rangle\leq\bigg \langle\hat{\theta},\left(\frac{C_2(\varepsilon\langle\xi\rangle^4+\lambda\langle\xi\rangle^2)+\frac{\kappa}{m^2}}{C_2(\varepsilon\langle\xi\rangle^4+\lambda\langle\xi\rangle^2)+1+\frac{\kappa}{m^2}}\right)\hat{\theta}\bigg\rangle_{L^2},\label{eq:i}
    \end{align}
where $\hat{\theta}$ is the Fourier transform of $\theta$.
        \item[(ii)] If $\|\rho_x\|_{H^1(0,1)}<\infty$ and $0<m\leq|\rho_x|\leq M,$ almost everywhere, then there exists $C_3>0$ such that 
\begin{align}
& \kappa \left\|\left(\frac{M^2\varepsilon}{C_3}\Delta+I(\kappa+M^2)\right)^{-\frac{1}{2}}\theta\right\|^2_{L^2} \nonumber \\  \leq & \,\langle\theta,(I-B^{-1})\theta \rangle \leq 
\|\theta\|^2_{L^2}-\|(\varepsilon\Delta^2-\lambda \Delta+I(\kappa+M^2))^{-\frac{1}{2}}(\theta\rho_x) \|_{L^2}^2. \label{eq:ii}
\end{align}
        
\item[(iii)] If $\|\rho_x\|_{L^2(0,1)}<\infty$ and $0\leq m\leq|\rho_x|\leq M,$ almost everywhere, then
     \begin{align}
      &\|\theta\|^2_{L^2}-\|(\varepsilon\Delta^2-\lambda \Delta+I(\kappa+m^2))^{-\frac{1}{2}}(\theta\rho_x) \|_{L^2} \nonumber \\ 
    \leq&\langle\theta,(I-B^{-1})\theta \rangle \nonumber \\
      \leq& \label{eq:iii} \|\theta\|^2_{L^2}-\|(\varepsilon\Delta^2-\lambda \Delta+I(\kappa+M^2))^{-\frac{1}{2}}(\theta\rho_x) \|_{L^2}\,.  
    \end{align}

\end{enumerate} 
\end{proposition}
\begin{proof}
Let $\psi\in H^2(0,1).$ 
We begin by proving \textit{(i)}.  Since $\rho_x\in H^2(0,1)$ and $|\rho_x|\geq m,$ we use \cite[Thm.~7.4]{behzadan2021multiplication} to say that the map $M_{\rho_x}(f)=f\rho_x$ is a bijective bounded linear map from $H^2(0,1)$ back to itself. As a consequence, the multiplication by $\frac{1}{\rho_x}$ is also bounded and there exist constants $C_1$ and $C_2$ such that 
\begin{align*}
\varepsilon\left\|\left(\frac{\psi}{\rho_x}\right)_{xx}\right\|^2_{L^2}+\lambda\left\|\left(\frac{\psi}{\rho_x}\right)_{x}\right\|^2_{L^2} &\leq C_2\left(\varepsilon\|\psi_{xx}\|^2_{L^2}+\lambda\|\psi_{x}\|^2_{L^2}\right)\,, \\
C_1\left(\varepsilon\|\psi_{xx}\|^2_{L^2}+\lambda\|\psi_{x}\|^2_{L^2}\right)&\leq\varepsilon\left\|\left(\frac{\psi}{\rho_x}\right)_{xx}\right\|^2_{L^2}+\lambda\left\|\left(\frac{\psi}{\rho_x}\right)_{x}\right\|^2_{L^2}.
\end{align*}
By sufficiently redefining  $C_1,C_2$, we can bound the value of the paring $\langle\psi,B\psi\rangle$ by
\begin{align}
    &{C_1}\left(\varepsilon\|\psi_{xx}\|^2_{L^2}+\lambda\|\psi_x\|^2_{L^2}\right)+\left(1+\frac{\kappa}{M^2}\right)\|\psi\|_{L^2}^2  \nonumber \\
    \leq&\langle\psi,B\psi\rangle \nonumber \\ \leq& {C_2}\left(\varepsilon\|\psi_{xx}\|^2_{L^2}+\lambda\|\psi_{x}\|^2_{L^2}\right)+\left(1+\frac{\kappa}{m^2}\right)\|\psi\|_{L^2}^2. \label{eq:B_ineq}
\end{align}
Denote by $B_m$ the operator where $$B_mv={C_2}\left(\varepsilon v_{xxxx}-\lambda v_{xx}\right)+{v}\left(1+\frac{\kappa}{m^2}\right)\,,\quad  \forall v\,,$$ and by $B_M$ the operator where
$$B_Mv={C_1}\left(\varepsilon v_{xxxx}-\lambda v_{xx}\right)+{v}\left(1+\frac{\kappa}{M^2}\right)\,,\quad  \forall v.$$
Using operators $B_m$ and $B_M$ we can rewrite the relation in~\eqref{eq:B_ineq} as 
$$
B_M\preceq B\preceq B_m.
$$ Similarly, we get 
$$
I-B_M^{-1}\preceq I-B^{-1}\preceq I-B_m^{-1}.
$$Using the Fourier transform, we can bound $\langle\theta,(I-B^{-1})\theta\rangle$ in the frequency domain and obtain~\eqref{eq:i}. \medskip
We move to the proof of \textit{(ii)}. Since $\rho_x$ no longer belongs to $H^2$, we cannot guarantee a bound of the form $\|(\frac{\psi}{\rho_x})_{xx}\|_{L^2}\leq C\|\psi\|_{L^2}^2$. Therefore, getting an upper bound as in part \textit{(i)} would not be possible without further assumptions. Recall $$L=\varepsilon\Delta^2-\lambda\Delta+I(\kappa+\rho_x^2).$$ Similarly, as before, we get the following relation
$$
 I-M_{\rho_x}L^{-1}_{\rho}M_{\rho_x}\preceq I-M_{\rho_x}L^{-1}_MM_{\rho_x}
$$ where 
$$
L_M=\varepsilon\Delta^2-\lambda\Delta+I(\kappa+M^2).
$$ 
After computing $\langle\theta,(I-\rho_xL^{-1}_M\rho_x)\theta\rangle$, we get 
$$
\langle\theta,(I-M_{\rho_x}L^{-1}_MM_{\rho_x})\theta\rangle=\|\theta\|^2_{L^2}-\|(\varepsilon\Delta^2-\lambda \Delta+I(\kappa+M^2))^{-\frac{1}{2}}(\theta\rho_x) \|_{L^2}^2.
$$
For the lower bound, we use the Poincar\'e--Wirtinger inequality \cite{attouch2014variational}  to conclude that there exists a constant $C_3$ such that 
$$
\frac{1}{C_3}\left\|\left(\frac{\psi}{\rho_x}\right)_{x}\right\|^2_{L^2}\leq \left\| \left(\frac{\psi}{\rho_x} \right)_{xx}\right\|^2_{L^2}.
$$  Furthermore, we can use the fact that multiplication with $\rho_x$ is bounded in $H^1(0,1)$ \cite[Thm.~7.4]{behzadan2021multiplication}
to get 
$$
\frac{\varepsilon}{\tilde C_3}\|\psi_x\|_{L^2}^2\leq\frac{\varepsilon}{C_3}\left\|\left(\frac{\psi}{\rho_x}\right)_{xx}\right\|^2_{L^2}.
$$ By appropriately adjusting the constant $C_3$, we get 
\begin{align*}
    \langle\psi,B\psi\rangle\geq \left(\lambda+\frac{\varepsilon}{C_3}\right)\|\psi_x\|^2_{L^2}+\left(1+\frac{\kappa}{M^2}\right)\|\psi\|_{L^2}^2 =: \langle \psi \,, H_M \psi \rangle,
\end{align*}
where the operator $H_M$ is defined such that $H_M v =  -\left(\lambda+\frac{\varepsilon}{C_3}\right) v_{xx}  + \left(1+\frac{\kappa}{M^2}\right) v$.
Using the Fourier transform, we can compute
$$
\langle\theta, (I-H_M^{-1})\theta \rangle=\left\langle\hat \theta,\frac{(\frac{\varepsilon}{C_3}+\lambda)\xi^2+\frac{\kappa}{M^2}}{(\frac{\varepsilon}{C_3}+\lambda)\xi^2+1+\frac{\kappa}{M^2}}\hat\theta\right\rangle\geq\kappa \left\|\left(\frac{M^2\varepsilon}{C_3}\Delta+I(M^2+{\kappa})\right)^{-\frac{1}{2}}\theta\right\|^2_{L^2}.
$$ Putting everything together, we obtain the chain of inequality~\eqref{eq:ii}. \medskip

For the last case \textit{(iii)}, since we do not have any bounds other than on the norm of $\rho_x$, we can only use an idea similar to the upper bound proof for part \textit{(ii)}. This gives us the chain of inequalities
$$
  I-M_{\rho_x}L^{-1}_{m}M_{\rho_x}\preceq I-M_{\rho_x}L^{-1}M_{\rho_x}\preceq I-M_{\rho_x}L^{-1}_MM_{\rho_x},
$$ 
where 
$$
L_m=\varepsilon\Delta^2-\lambda\Delta+I(\kappa+m^2).
$$ 
As a result, we get~\eqref{eq:iii}.
\end{proof}

\subsection{Spectral properties of the Hessian Operator}
In this subsection, we aim to describe the effect of hyperparameters $\kappa,\lambda$, and $\varepsilon$ on the local convexity of the squared HV metric through the analysis of the Hessian operator given in~\eqref{eq:Binverse}. 

Proposition~\ref{loewner} suggests that if $\rho_x\in H^2(0,1)$, then it is possible to get the chain of 
The denominator of the last term in the inequality~\eqref{eq:i} closely resembles the form of the inverse Laplacian. This leads us to define a new norm 
\begin{equation}
\label{normequiv}
    \|\theta\|_{H_{\kappa,\lambda,\varepsilon, M}^{2s}(0,1)}: =\|(\varepsilon\Delta^2-\lambda\Delta+I(\kappa+M^2))^{\frac{s}{2}}\theta\|_{L^2(0,1)}\,,
\end{equation}
which is equivalent to the usual Sobolev $H^{2s}(0,1)$ norm.
The next two propositions connect the HV metric and the $L^2$ norm. 
\begin{proposition}If $\theta,\rho\in H^1(0,1),$ then
$d_{\text{HV}}(\rho+s\theta,\rho)$ locally equals to the $L^2$ norm of $s\|\theta\|_{L^2}$ for any choice of hyperparameters, i.e.,
\[
\lim_{s \to 0}    \frac{d_{\text{HV}}^2(\rho,\rho+s\theta)}{s^2} = \|\theta\|^2_{L^2}\qquad \forall \kappa,\lambda,\varepsilon\,,
\]
if and only if $\rho_x=0.$ 
\end{proposition}

\begin{proof}
We first assume that $\rho_x=0.$ Due to \cite[Prop.~2.1]{han2024hv} we can assume that $\rho=0$. From the optimality conditions, we see that $v=0,$ and $z(x,t)=s\theta(x) $ and 
$$
d_{\text{HV}}^2(\rho, \rho + s\theta)=s^2\int_0^1\int_0^1\theta(x)dxdt=s^2\|\theta\|^2_{L^2}.
$$
The other direction follows from the local expansion of the metric. In particular, because  
$$
\lim_{s\to0}\frac{d_{\text{HV}}^2(\rho,\rho+s\theta)}{s^2}=\|\theta\|_{L^2}^2-\int_0^1(\theta\rho_x)L^{-1}(\rho_x\theta) dx=\|\theta\|^2_{L^2}\quad \forall\theta,
$$ it must hold that for all $\theta$ that
$$
\int_0^1(\theta\rho_x)L^{-1}(\rho_x\theta) dx=0\,,
$$ 
which only happens if $\rho_x=0.$ 
\end{proof}

Although the HV metric coincides with the $L^{2}$ norm only in the local regime where $\rho_{x}=0$, one can still approximate $L^{2}$ behavior by selecting suitable values for the parameters $\varepsilon$, $\lambda$, and $\kappa$.

\begin{proposition}
\label{L2kappa}
If $\rho,\theta\in H^1(0,1)$, then the following asymptotic relations hold
\begin{enumerate}
    \item For fixed $\lambda\geq0$ and $\varepsilon>0$ we have $$
\lim_{\kappa\to\infty}\lim_{s\to 0}\frac{d_{\text{HV}}^2(\rho,\rho+s\theta)}{s^2}=\|\theta\|^2_{L^2}.
$$\item For fixed $\kappa,\varepsilon>0$ we have$$
\lim_{\lambda\to\infty}\lim_{s\to 0}\frac{d_{\text{HV}}^2(\rho,\rho+s\theta)}{s^2}=\|\theta\|^2_{L^2}.
$$\item For fixed $\lambda\geq0$ and $\kappa>0$ we have$$
\lim_{\varepsilon\to\infty}\lim_{s\to 0}\frac{d_{\text{HV}}^2(\rho,\rho+s\theta)}{s^2}=\|\theta\|^2_{L^2}.
$$
\end{enumerate}

\end{proposition}

\begin{proof}

Property $(iii)$ in Proposition \ref{loewner} implies the following upper and lower bounds on the spectrum 
\begin{equation}
    \label{eqkle}
\|\theta\|^2_{L^2}-\|(\varepsilon\Delta^2-\lambda \Delta+\kappa I)^{-\frac{1}{2}}(\theta\rho_x) \|_{L^2}\leq\lim_{s\to 0}\frac{d_{\text{HV}}^2(\rho,\rho+s\theta)}{s^2}\leq\|\theta\|^2_{L^2}.
\end{equation} We will show that 
$$
\lim_{\kappa \to\infty}\|(\varepsilon\Delta^2-\lambda \Delta+I\kappa)^{-\frac{1}{2}}(\theta\rho_x)\|_{L^2}^2=0.
$$ Since $\theta\rho_x\in L^2(0,1)$, we use the Plancherel identity to get 
$$
\lim_{\kappa \to\infty}\|(\varepsilon\Delta^2-\lambda \Delta+I\kappa)^{-\frac{1}{2}}(\theta\rho_x)\|_{L^2}^2=\lim_{\kappa \to\infty}\int_{\R}\frac{ |\widehat{\theta\rho_x}(\xi)|^2}{\varepsilon\xi^4+\lambda\xi^2+\kappa} d\xi =0,
$$ where the last limit follows from the dominated convergence theorem. By letting $\kappa\rightarrow \infty$ in \eqref{eqkle}, we have shown claim 1. The other two claims are proven similarly.
\end{proof}
So far, we have only discussed the case when the hyperparameters go to infinity. We now focus on the other extreme when the hyperparameters $\varepsilon,\lambda$ and $\kappa$ approach $0$. This is our main result of this section.
\begin{theorem}
\label{main1}
Let $\rho_x \in H^2(0,1)$ satisfy $0 < m \leq |\rho_x| \leq M$ almost everywhere, and let $\theta \in L^2(0,1)$. If there exists $T > 0$ such that $\supp(\mathcal{F}(\theta)) \subset [-T, T]$, where $\mathcal{F}$ denotes the Fourier transform operator, then there exist constants $\tilde{C}_1, \tilde{C}_2 > 0$ and hyperparameters $\kappa_0, \lambda_0, \varepsilon_0$ such that for all $\kappa \in (m^2, \kappa_0)$, $\lambda \leq \lambda_0$, and $\varepsilon \leq \varepsilon_0$, the following inequality holds:
$$
0<\tilde C_1\|\theta\|_{H^{-2}}^2\leq\langle\theta,(I-B^{-1})\theta \rangle\leq \tilde C_2\|\theta\|^2_{H^{-2}},
$$
 where $\kappa_0,\lambda_0$ and $\varepsilon_0$ depend only on $T.$
\end{theorem}

\begin{proof}
Due to $(i)$ in Proposition \ref{loewner}, we get the following bounds for the Hessian operator $I-B^{-1}$,
    \begin{align}
    \label{eq:14}
    \kappa\|(C_1M^2\varepsilon\Delta^2-C_1M^2\lambda (\Delta+(\kappa+M^2)I))^{-\frac{1}{2}}\theta \|^2_{L^{2}}&\leq\langle\theta,(I-B^{-1})\theta \rangle \leq\langle\theta,(I-B_m^{-1})\theta\rangle,
\end{align}
where $$B_mv={C_2}\left(\varepsilon v_{xxxx}-\lambda v_{xx}\right)+{v}\left(1+\frac{\kappa}{m^2}\right).$$

We split $B_m$ into a sum of an operator $D$ and the identity operator $I$:
$$
B_m=\left({C_2}\left(\varepsilon\frac{d^4}{dx^4}-\lambda\frac{d^2}{dx^2}\right)+\frac{\kappa}{m^2} I \right)+I=D+I\,.
$$ 
Since when $\kappa >m^2$, the norm of $\|D^{-1}\|_{L^2\to L^2}<1$, we can expand the operator $$B_m^{-1}=D^{-1}(I+D^{-1})^{-1}$$ into an absolutely convergent Neumann series 
$$
D^{-1}(I+D^{-1})^{-1}=D^{-1}(I-D^{-1}+D^{-2}-\cdots)=D^{-1}-D^{-2}+D^{-3}-\cdots.
$$Since this is an alternating series, we can use its expansion to get the following relation:
$$
I-B^{-1}\preceq I-D^{-1}+D^{-2}-D^{-3}+\cdots\preceq I-D^{-1}+D^{-2}-D^{-3}+D^{-4}=: \tilde D
$$ 

Observe that the operator $D$ and its inverse are both diagonalizable in the Fourier domain. We denote by $\psi(\xi)=\frac{1}{{C_2}(\varepsilon\xi^4+\lambda\xi^2)+\frac{\kappa}{m^2}}$ the Fourier multiplier of $D^{-1}$ 
and compare it with $\tilde\psi(\xi):=1-\psi(\xi)+\psi^2(\xi)-\psi^3(\xi)+\psi^4(\xi)$, the Fourier multiplier of $\tilde D$. Since the polynomial $1-x^2+x^3+x^4$ is even and convex, it has a unique minimum $y_0<1$. If we choose $\kappa>m^2$ such that 

\begin{equation}
\label{eq:15}
    {y_0}\leq\frac{1-(\frac{m^2}{\kappa})^4}{1+\frac{m^2}{\kappa}}\leq\frac {m^2} {\kappa},
\end{equation} 
then $$\tilde\psi''(0)=\frac{2C_2\lambda m^4}{\kappa^2}\left(1-\frac{2m^2}{\kappa}+\frac{3m^4}{\kappa^2}-\frac{4m^6}{\kappa^3}\right)<0\text{ \quad and \quad } \psi''(0)=-\frac{2C_2\lambda m^4}{\kappa^2}.$$ The function $g(\kappa)=\left(1-\frac{2m^2}{\kappa}+\frac{3m^4}{\kappa^2}-\frac{4m^6}{\kappa^3}\right)$ is strictly increasing on the interval $[m^2,\frac{m^2}{y_0}]$ 
with the minimum achieved at $\kappa=m^2,$ with the value of $-2,$ and the maximum achieved at $\kappa=y_0,$ with the value of $0.$ Since the function $g$ is continuous, there exists $\kappa_0,$ such that $\forall \kappa\in(m^2,\kappa_0]: g(\kappa)<-1.$ Comparing $\tilde \psi''(0)$ with $\psi''(0)$ for $m^2< \kappa\leq\kappa_0$, we have
 
\begin{equation}
\label{eq:16}
    \tilde\psi''(0)<\psi''(0)<0,
\end{equation} 
which suggests that, at least in the neighborhood of $0,$ $\tilde\psi$ is more concave than $\psi.$ 

We are now ready to complete the proof for the upper bound. Note that \eqref{eq:15} and \eqref{eq:16} suggest that around $\xi=0$ , $\tilde\psi$ can be bounded above by $\psi$.  On the other hand, as $\xi\rightarrow \infty$, $\psi\rightarrow 0$ while $\tilde\psi\rightarrow 1$. Our goal will be to classify all the points $\xi,$ such that $\tilde\psi(\xi)=\psi(\xi).$ That is, we wish to find the roots of the equation $$
1-2\psi(\xi)+\psi^2(\xi)-\psi^3(\xi)+\psi^4(\xi)=0.
$$ 
The left-hand side function is the composition of $h(x)=1-2x+x^2-x^3+x^4,$ and the multiplier $\psi(\xi).$ The function $h(x)$ has two real roots, $1$ and $r_1$, and two complex roots. When $\kappa=m^2$, the solution to $\psi(\xi)=1$ occurs at $\xi=0$. When $\kappa>m^2$, $\psi(\xi)=1$ has $4$ imaginary and $0$ real solutions, so we do not consider the real root $1$ of $h(x)$.

Using the other real root $r_1$ of $h(x)$, we can solve the equation $\psi(\xi)=r_1$ to compute the frequencies at which $\psi(\xi)=\tilde\psi(\xi)$. Since $\psi(\xi)$ is monotonically decreasing and even, we denote the two real solutions to $\psi(\xi)=r_1$ as $\xi_{0}$ and $\xi_1$ where $\xi_1 = - \xi_0 <0 $.

By choosing $\lambda_0,\varepsilon_0>0$ small enough such that $\xi_0\geq T$,  we get the following inequality for all $m^2<\kappa\leq\kappa_0$, $\lambda\leq\lambda_0$,
$\varepsilon\leq\varepsilon_0$:

\begin{align*}
    \kappa\|(C_1M^2\varepsilon\Delta^2-C_1M^2\lambda (\Delta+(\kappa+M^2)I))^{-\frac{1}{2}}\theta \|^2_{L^{2}}&\leq \langle\theta,(I-B^{-1})\theta \rangle\\&\leq\left\| \left(C_2(\varepsilon\Delta^2-\lambda\Delta)+\frac{\kappa}{m^2}I\right)^{-\frac{1}{2}}\theta\right\|^2_{L^2},
\end{align*}
where the left-hand side is due to~\eqref{eq:14} and the right-hand side is due to~\eqref{eq:15},~\eqref{eq:16} and the fact that $\psi(\xi)\geq\tilde\psi(\xi)$ on the interval $[-T,T]$. 

Using the definition~\eqref{normequiv} for the case $s=-1$ and its equivalence to the standard Sobolev norm $H^{-2}(0,1)$, we have
  \begin{align*}
  \tilde C_1\|\theta\|_{H^{-2}}^2\leq\langle\theta,(I-B^{-1})\theta \rangle\leq \tilde  C_2\|\theta\|^2_{H^{-2}}\,.
\end{align*}
 \end{proof}
\begin{remark}
The proof of Theorem~\ref{main1} reveals insights that go well beyond the statement itself. By expanding the operator $B^{-1}$ into a Neumann series and approximating it with a finite number of terms, we see that for small enough $\kappa_0$, the Fourier transform of $\tilde D$  will have a distinct bell shape around $\xi=0$, comparable to that of $-\Delta ^2.$ The width of the bell shape can be controlled by $\varepsilon $ and $\kappa$.
\end{remark}

Theorem~\ref{main1} shows that the linearization of the squared HV metric is bounded above and below by constant multiples of the $H^{-2}$ norm, provided the input signal has a compactly supported Fourier transform and the hyperparameters are chosen sufficiently small. The parameter $\varepsilon$ weights the biharmonic operator $\Delta^{2}$ in the expression $I-B^{-1}$ and is therefore responsible for the $H^{-2}$-like behavior. If $\varepsilon$ were set to zero, this term would vanish, the factor $\lambda\Delta$ would dominate, and the metric would exhibit $H^{-1}$-type behavior instead. Although $\varepsilon$ cannot be exactly zero, the next corollary shows that the HV metric is expected to exhibit $H^{-1}$-type behavior as $\varepsilon\to 0$.

\begin{corollary}
\label{corollaryH1}
Under the same assumptions as Theorem~\ref{main1}, we have 
   $$
    \tilde C_1\|\theta\|_{H^{-1}}^2\leq\lim_{\varepsilon\to 0}\langle\theta,(I-B^{-1})\theta \rangle\leq \tilde C_2\|\theta\|^2_{H^{-1}}\,. 
    $$
\end{corollary}
\begin{proof}

Due to Proposition~\ref{loewner} and Theorem~\ref{main1}, we have the following bounds 
\begin{align*}
\kappa\lim_{\varepsilon\to 0}\int_{R}\frac{|\hat\theta(\xi)|^2}{C_1M^2\varepsilon\xi^4+C_1M^2\lambda\xi^2+\frac{\kappa}{M^2}}d\xi
\leq & \lim_{\varepsilon\to 0}\langle\theta,(I-B^{-1})\theta \rangle\\
\leq & \lim_{\varepsilon\to 0}\int_{R}\frac{|\hat\theta(\xi)|^2}{\frac{1}{C_2}\left(\varepsilon\xi^4+\lambda\xi^2\right)+\frac{\kappa}{m^2}}d\xi.
\end{align*}
Since $\theta \in H^1(0,1)$,

we can use Lebesgue's dominated convergence theorem to pass the limit inside the integral, from which the corollary follows.
\end{proof}

We end this section by summarizing Proposition~\ref{L2kappa} and Theorem~\ref{main1} and providing numerical examples describing the local limiting behavior of the squared HV metric based on the choice of hyperparameters $\kappa,\lambda$, and $\varepsilon.$

Proposition~\ref{L2kappa} shows that enlarging any of the hyperparameters $\kappa$, $\lambda$, or $\varepsilon$ pushes the HV metric toward local $L^{2}$ behavior, whereas Theorem~\ref{main1} indicates that small values of these parameters yield a more transport‐like local structure. These results align with our intuition for how the HV metric should behave. Recall that the squared HV distance is defined through the following action functional.
\begin{equation*}
   A_{\kappa, \lambda, \varepsilon}(f,v,z) =\frac12 \int_0^1 \int_0^1   \kappa v^2 + \lambda v_x^2  + \varepsilon v_{xx}^2 + z^2    \, d x d t.
\end{equation*}
When the velocity‐penalizing hyperparameter $\kappa$ is large, the cost of transport via the velocity field $v$ dominates, so the optimal path typically sets $v\approx 0$ and pays instead for the vertical deformation carried by $z$.  Conversely, when $\kappa$, $\lambda$, and $\varepsilon$ are all small, horizontal deformations become relatively inexpensive, making transport through $v$ more advantageous than vertical adjustments.

We illustrate Proposition~\ref{L2kappa}, Theorem~\ref{main1} and Corollary \ref{corollaryH1} numerically. We regard a pair of two Ricker wavelet signals as $g_s(t)$ and $f(t)$ where $g_s(t)=f(t-s)$ as shown in Fig.~\ref{rickershift}.
\begin{figure}[ht]
    \centering
    \includegraphics[width=0.5\linewidth]{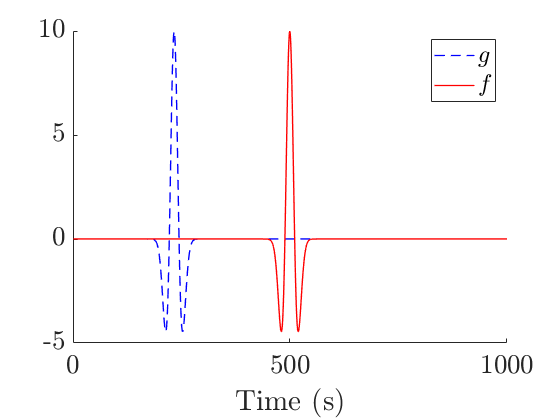}
    \caption{Original signal $f$ and shifted signal $g_s.$}
    \label{rickershift}
\end{figure}
We examine the curve
\[
J(s)\;=\;d_{\mathrm{HV}}^{2}\bigl(g_{s},f\bigr)
\]
as a function of the translation parameter $s$.

In Fig.~\ref{convexpara}, we analyze three representative hyperparameter regimes for the HV metric and compare each to the squared $L^{2}$ distance between $f$ and $g_{s}$ (see Fig.~\ref{fig:convex1}):
\begin{enumerate}
    \item \textbf{$L^2$ regime.}  
          Setting $\kappa = 10^{10}$ with $\lambda = \varepsilon = 1$ illustrates the $L^{2}$‐like behavior predicted by Proposition~\ref{L2kappa}; see Fig.~\ref{fig:convex1} and Fig.~\ref{fig:convex2}.
          
    \item \textbf{$H^{-2}$ regime.}  
          Choosing all three hyperparameters small but with $\varepsilon \gg \lambda$ produces the $H^{-2}$‐type behavior indicated by Theorem~\ref{main1}; see Fig.~\ref{fig:convex3}.
          
    \item \textbf{$H^{-1}$ regime.}  
          Keeping the hyperparameters small while reversing the hierarchy to $\lambda \gg \varepsilon$ yields the $H^{-1}$‐like behavior described in Corollary~\ref{corollaryH1}; see Fig.~\ref{fig:convex4}.
\end{enumerate}
When $s$ is large, we observe a constant behavior in Fig.~\ref{fig:convex3} and Fig.~\ref{fig:convex4}. This occurs because, when the original signal $f$ and the shifted signal $g_s$ are sufficiently far apart, the cost of modifying $v$ in the action functional \eqref{action} to align the two signals becomes greater than the cost of adjusting $z$. As a result, vertical shifts become the more cost-effective option. The transition point where this behavior shifts happens much later in Fig.~\ref{fig:convex3} compared to Fig.~\ref{fig:convex4}, due to the smaller weight assigned to the velocity and its derivatives in the action functional.

\begin{figure}[!ht]
    \centering
    \subfloat[$L^2$ misfit]{\includegraphics[width=0.4\linewidth]{  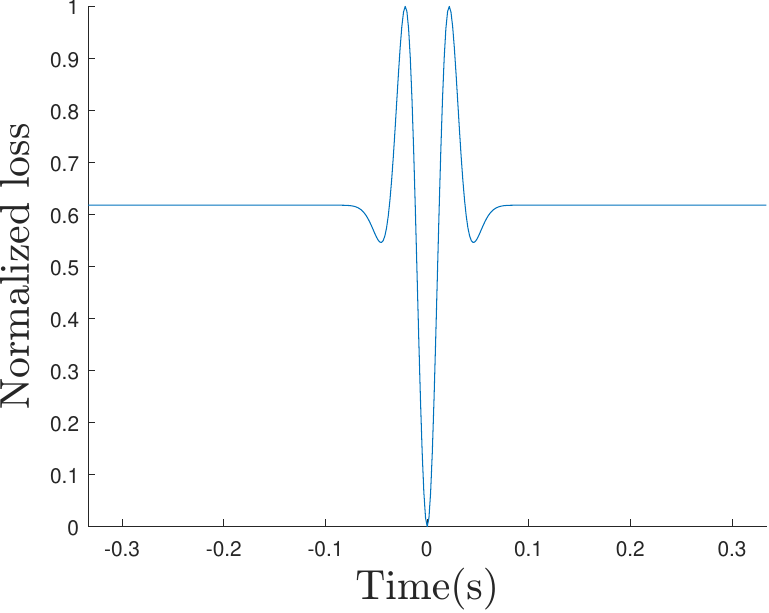}\label{fig:convex1}}
    \hspace{0.05\linewidth}
    \subfloat[HV misfit in the $L^2$ regime]{\includegraphics[width=0.4\linewidth]{  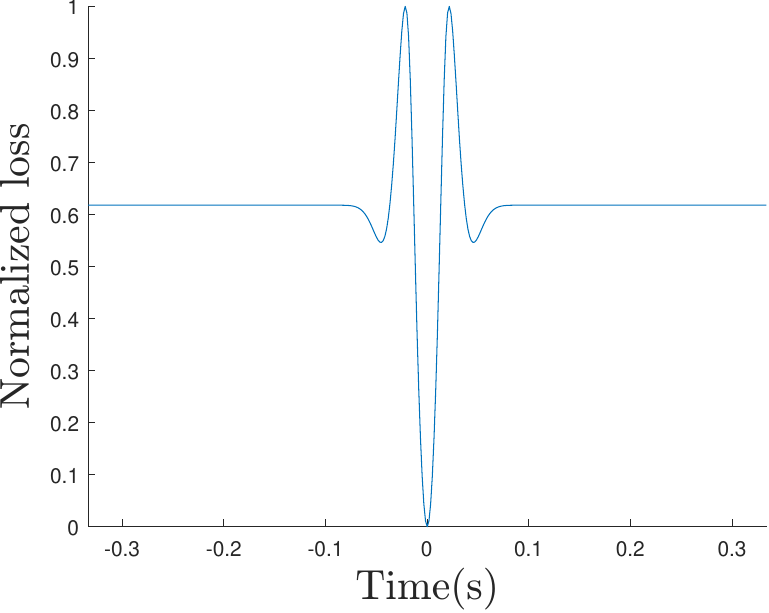}\label{fig:convex2}}\\
    \subfloat[HV misfit in the $H^{-2}$ regime]{\includegraphics[width=0.4\linewidth]{  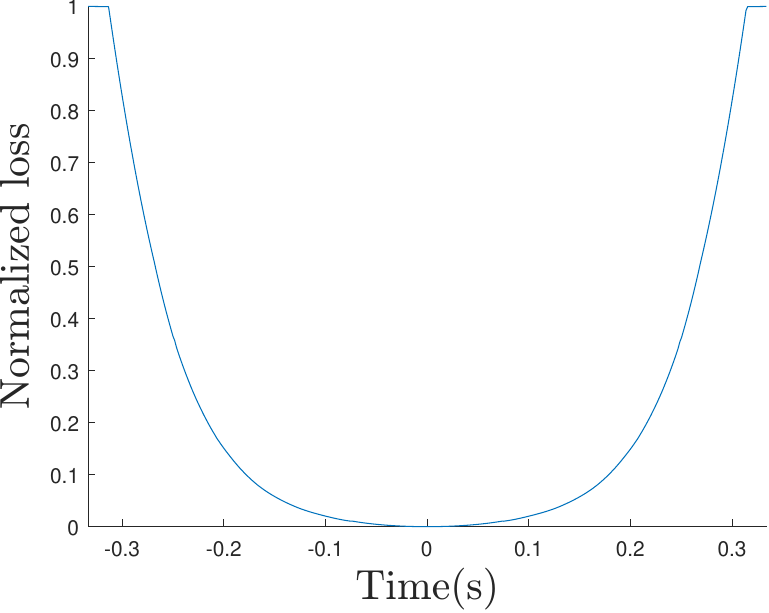}\label{fig:convex3}}
    \hspace{0.05\linewidth}
    \subfloat[HV misfit in the $H^{-1}$ regime]{\includegraphics[width=0.4\linewidth]{  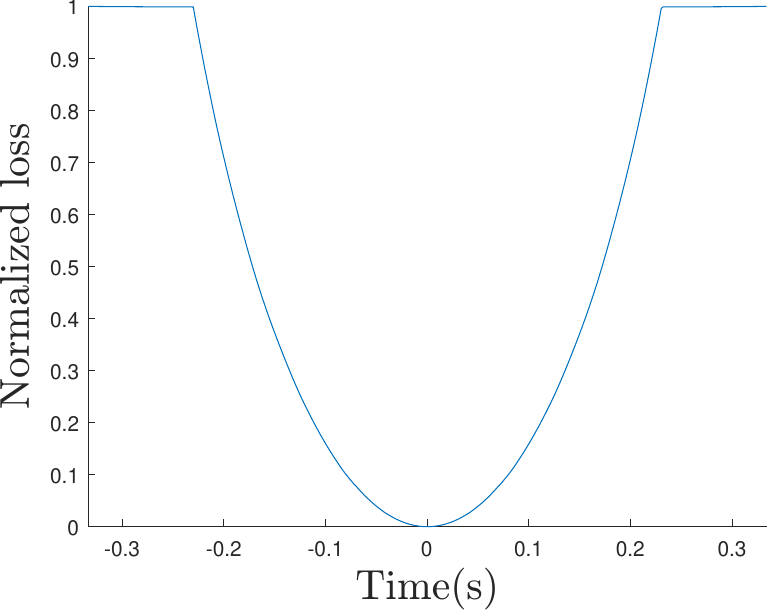}\label{fig:convex4}}
    \caption{The misfit between $f(x)$ and $g_s : = f(x-s)$ using (a) the $L^2$ norm, (b) the HV metric with $\kappa=\lambda=10^{5}$ and $\varepsilon=10$, (c) the HV metric with $\kappa=\lambda=10^{-5}$ and $\varepsilon=10^{-5}$, and (d) the HV metric with $\kappa=\varepsilon=10^{-5}$ and $\lambda=10$.}
    \label{convexpara}
\end{figure}

\section{Application of HV to Seismic Inversion}\label{numex}
In this section, we test the HV metric as a data misfit function to measure the discrepancy between synthetic wave data and observed wave data in time-domain FWI, where the goal is to recover the wave speed from boundary measurements of wave displacement. We focus on two standard benchmark velocity models: the Marmousi model and the BP salt model. The performance of the HV metric is compared against the $L^2$ norm and the 2-Wasserstein metric.
\subsection{Marmousi Model}
 
We invert the full Marmousi model, a standard benchmark for exploration geophysics since the 1980s. The model is based on a profile of the North Quenguela Trough in the Cuanza Basin in Angola \cite{versteeg1994marmousi}. The HV metric used in this inverse problem is computed with hyperparameters $\kappa=\lambda=10^{-4},\varepsilon=10^{-12}$ and the number of iterations is $15$. The $v$ component in the computation of the HV metric is initialized by matching anywhere from zero to six peaks; we refer to~\cite[Sec.~4.3]{han2024hv} for more details on the initialization. 

Fig.~\ref{fig:Marm1} represents the P-wave velocity of the true Marmousi model. Our inverse problem begins with an initial model obtained by smoothing the true velocity using a Gaussian filter with a variance of 30, as shown in Fig.~\ref{fig:Marm2}. We place 15 evenly spaced sources at a depth of 50 m and 307 receivers at the same depth. Discretization of the forward wave equation is 10 m in the $x$ and $z$ directions with 1 ms in time. The source is generated by a Ricker wavelet with a peak frequency of 15 Hz, and the acquisition time is 4.5 seconds. As the optimization algorithm, we use a quasi-Newton method called L-BFGS-B~\cite{LBFGS} for which we keep track of ten iterates in the memory for this optimization scheme. Inversions are terminated after 400 iterations. The final inversion result is presented in Fig.~\ref{fig:Marm3}.

\begin{figure}[!ht]
    \centering
    \subfloat[]{\includegraphics[width=0.3\linewidth]{  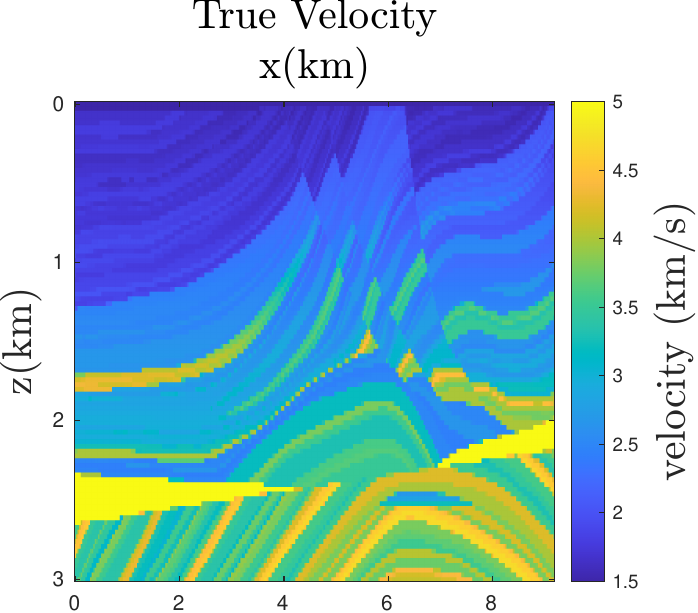}\label{fig:Marm1}} \hspace{0.05\linewidth}
    \subfloat[]{\includegraphics[width=0.3\linewidth]{  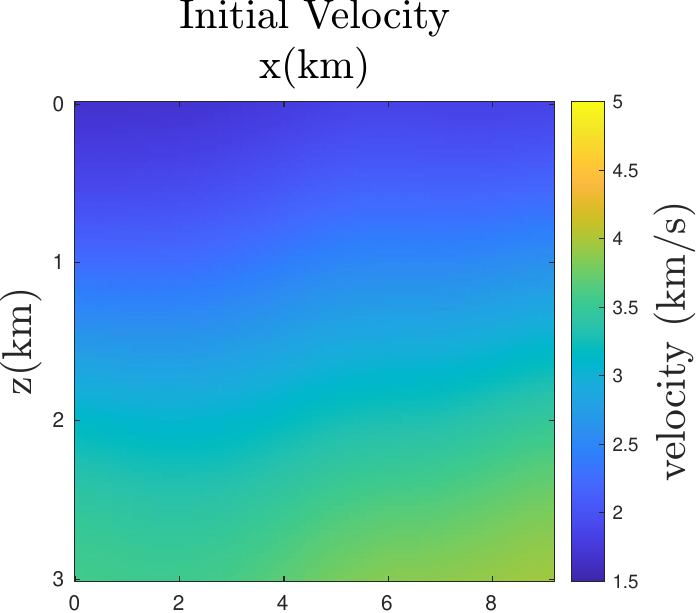}\label{fig:Marm2} }\hspace{0.05\linewidth}
    \subfloat[]{\includegraphics[width=0.30\linewidth]{  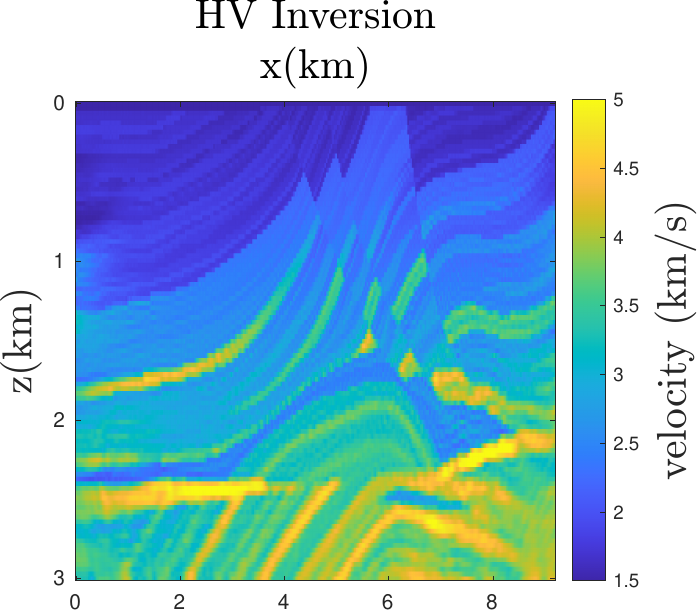}\label{fig:Marm3}} \hspace{0.025\linewidth}
    \subfloat[]{\includegraphics[width=0.30\linewidth]{  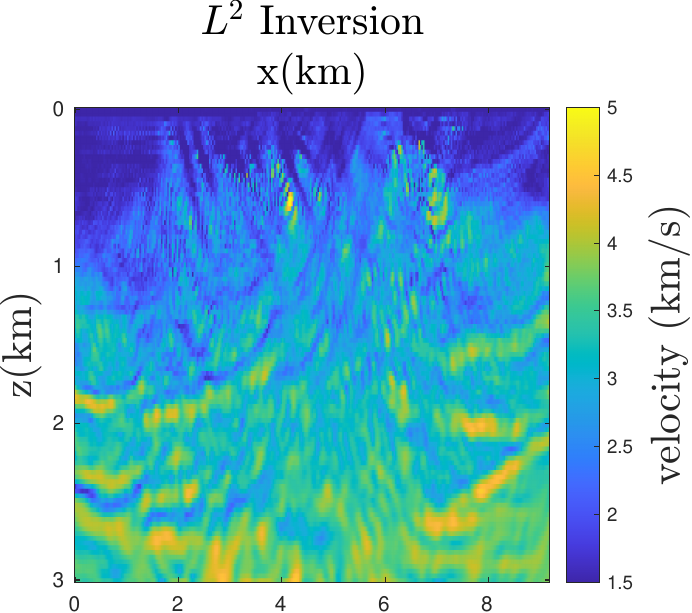}\label{fig:Marm5}}  \hspace{0.025\linewidth}
    \subfloat[]{\includegraphics[width=0.30\linewidth]{  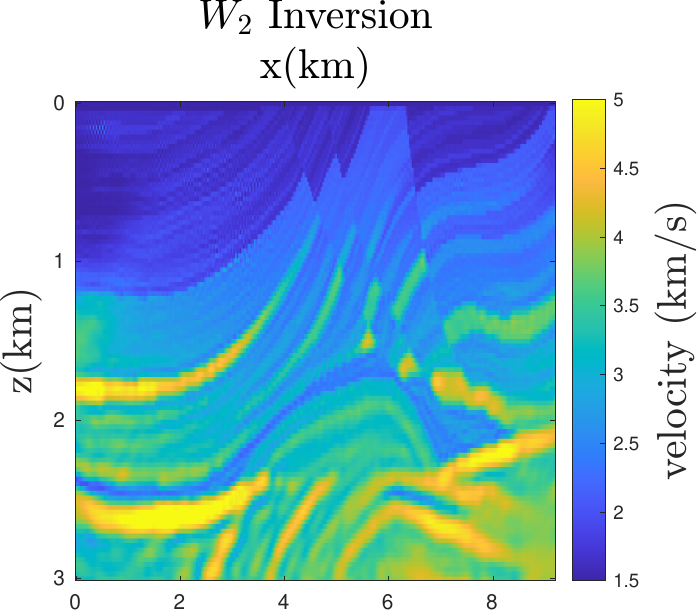}\label{fig:Marm7}} 
    \caption{Marmousi model inversion: (a) true velocity; (b) initial velocity; (c) HV inversion result; (d) $L^2$ inversion result; (e) $W_2$ inversion result}
    \label{MarmComp1}
\end{figure}

We compare the inversion results using the HV metric as the objective function against those using the $L^2$ norm and the $2$-Wasserstein ($W_2$) distance to measure the data misfit. Fig.~\ref{fig:Marm5} shows the inversion result using the traditional $L^2$ norm, while Fig.~\ref{fig:Marm7} shows the inversion result for the Wasserstein distance. We observe that the result of the $L^2$ norm exhibits spurious high-frequency artifacts, whereas the HV and $W_2$ methods correctly invert most of the details in the true model. When comparing the inversion results between the HV metric and the Wasserstein distance, we notice that the HV metric recovers finer details and avoids inconsistencies such as the banana-shaped feature seen in Fig.~\ref{fig:Marm7}, which occurs due to the data normalization step that turns the signed wave data into probability measures required to compute the Wasserstein distance.

 \subsection{Salt Model}
 In this section, we invert a different kind of velocity model, primarily using reflection waves. Fig.~\ref{fig:Salt1} represents the true velocity model, which is part of the 2004 BP benchmark \cite{salt2004}. In contrast to the Marmousi model explored in the previous subsection, the main challenge here is obtaining the delineation of the salt and recovering information on the sub-salt velocity variations. The HV metric in the inversion problem is computed using the hyperparameters $\kappa=\lambda=10^{-8},\varepsilon=10^{-4}$ with the number of iterations equal to $10$. The $v$ component in the computation of the HV metric is initialized by matching between 0 and 4 peaks. 
 
 The inversion begins with an initial model obtained by smoothing the true velocity model using a Gaussian filter, as shown in Fig.~\ref{fig:Salt2}. We place 11 equally spaced point sources at the top boundary of the domain, each producing a 15 Hz Ricker wavelet, and 375 receivers are equally spaced on the top of the domain. The total recording time is 4 seconds. The observed data are dominated by the reflection from the top of the salt inclusion, as also evident in the inversion results in Figs.~\ref{fig:Salt3} and \ref{fig:Salt4}. We observe that in the $L^2$-based reconstruction after 100 iterations (see Fig.~\ref{fig:Salt4}),  a wrong sub-layer was created below the salt-water interface, from which one might conclude a misleading interpretation of the subsurface medium property. On the other hand, the sub-layer in the HV inversion (see Fig.~\ref{fig:Salt3}) after 100 iterations is almost non-existent, which more closely resembles the true velocity model. 
 
\begin{figure}[!ht]
    \centering
    \subfloat{\includegraphics[width=0.40\linewidth]{  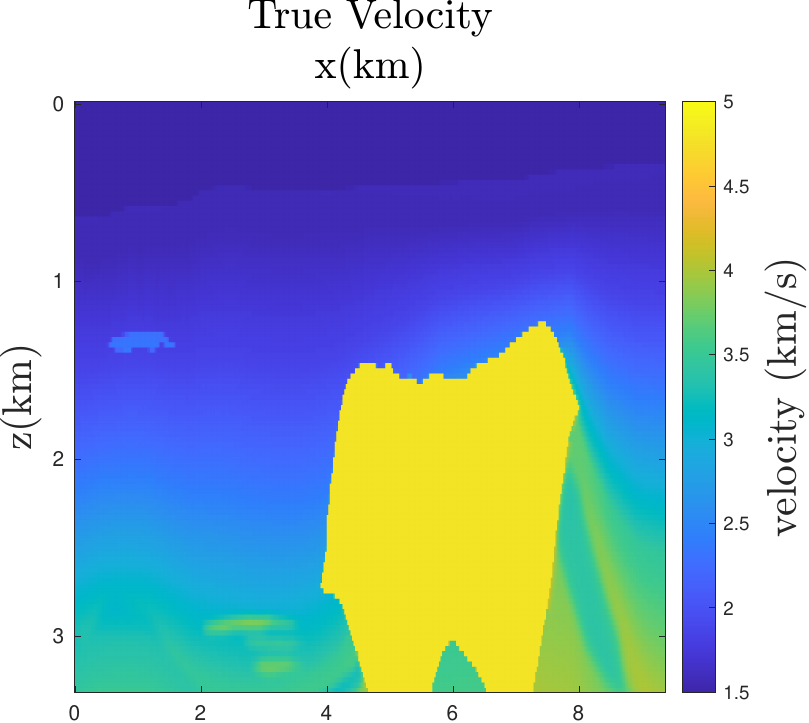}\label{fig:Salt1}}
    \hspace{0.05\linewidth}
    \subfloat[]{\includegraphics[width=0.40\linewidth]{  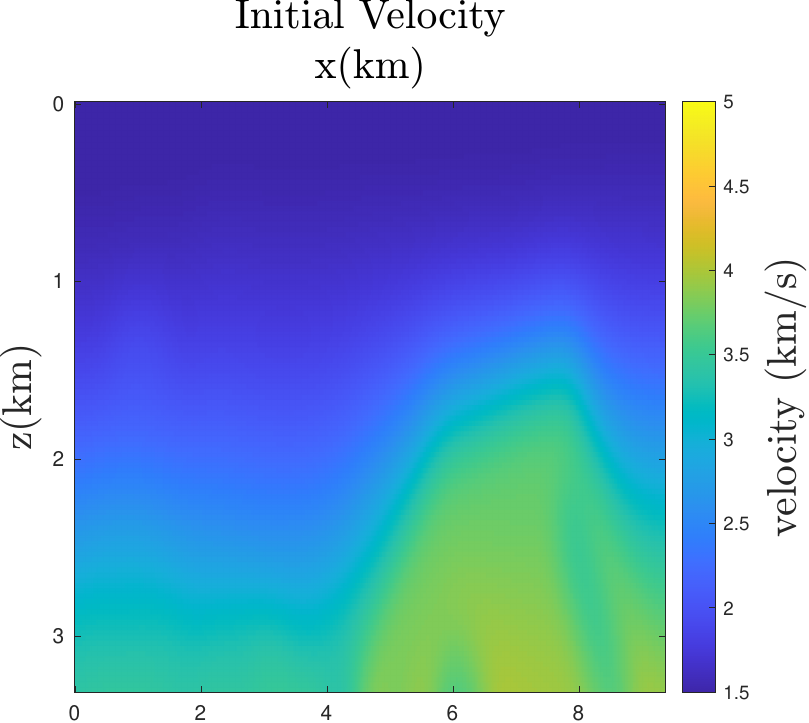}\label{fig:Salt2}}
    \label{fig:time}
    \subfloat[]{\includegraphics[width=0.40\linewidth]{  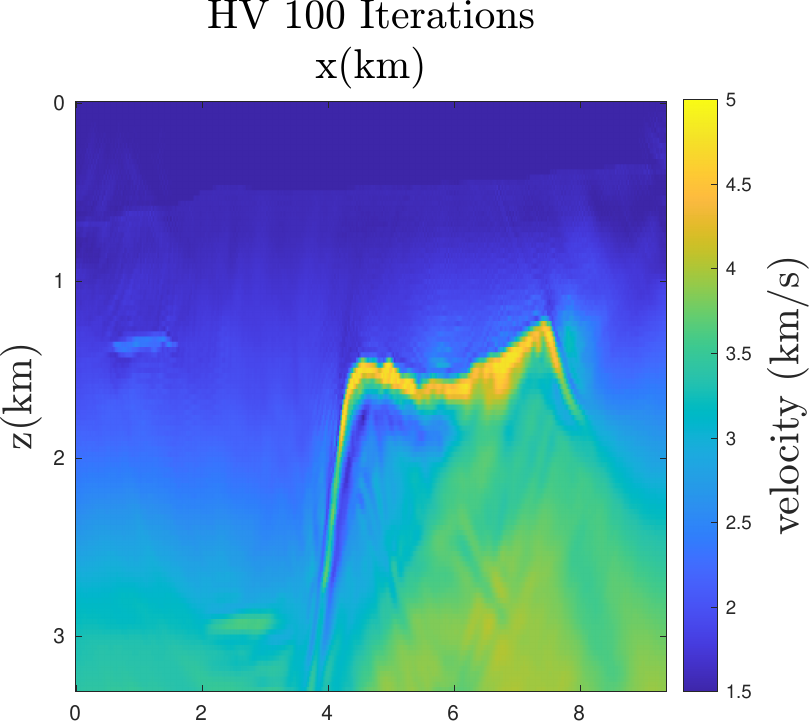}\label{fig:Salt3}}
    \hspace{0.05\linewidth}
    \subfloat[]{\includegraphics[width=0.40\linewidth]{  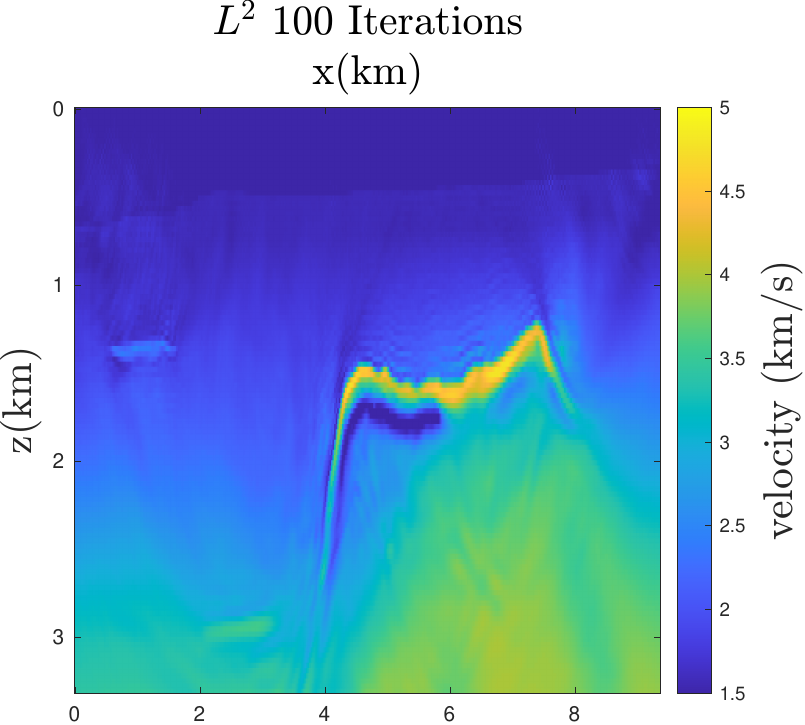}\label{fig:Salt4}}    \hspace{0.05\linewidth}
     \subfloat[]{\includegraphics[width=0.40\linewidth]{  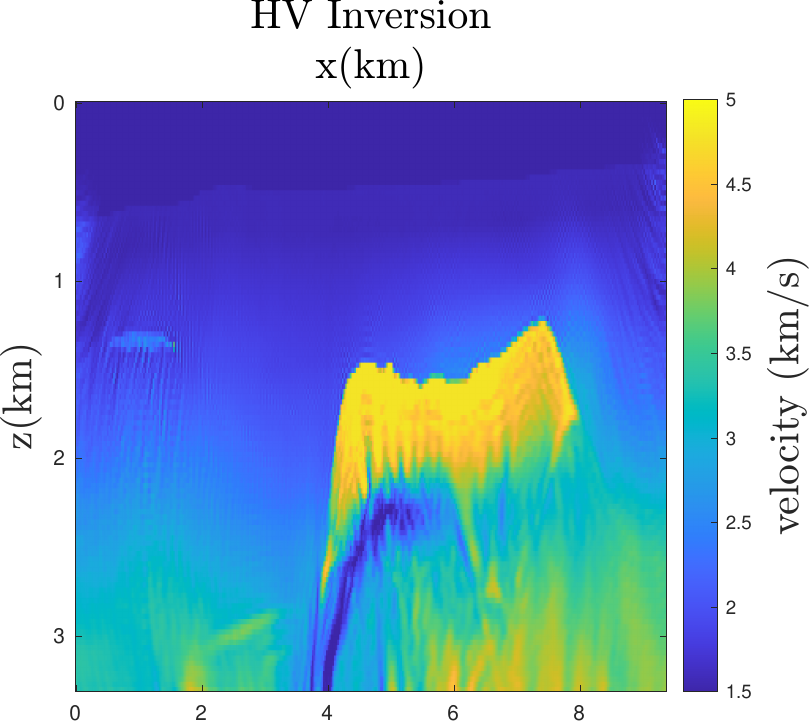}\label{fig:Salt5}}
    \hspace{0.05\linewidth}
    \subfloat[]{\includegraphics[width=0.40\linewidth]{  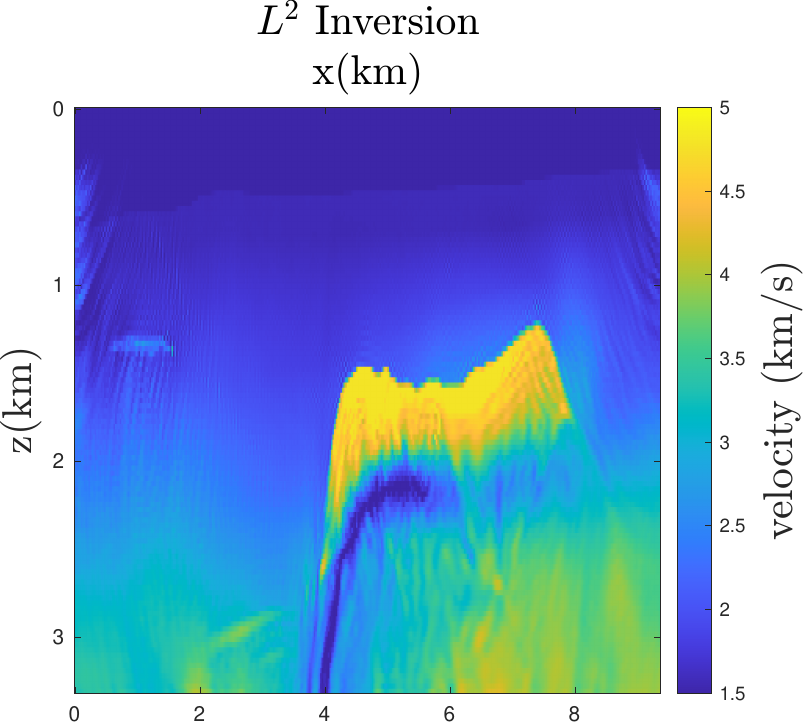}\label{fig:Salt6}}
    \caption{BP model inversion: (a) true velocity; (b) initial velocity; (c) HV result after 100 iterations of L-BFGS-B; (d) $L^2$ result after 100 iterations of L-BFGS-B; (e) HV inversion after 800 iteration; (f) $L^2$ inversion after 800 iterations}
\end{figure}

As the number of iterations increases, the differences between the reconstructions become more pronounced. Fig.~\ref{fig:Salt5} illustrates the HV-based reconstruction after $800$ iterations, which achieves better recovery of the salt body compared to the $L^2$ inversion after the same number of iterations (see Fig.~\ref{fig:Salt6}). Overall, our numerical observation shows that the inversion using the HV metric consistently recovers more of the salt model while producing fewer and less prominent sub-layers that the $L^2$-based approach.

\section{Conclusion}\label{conclusion}

This study establishes the HV metric as a powerful misfit function for full-waveform inversion, introducing the first transport-based distance capable of comparing \emph{signed} waveforms without requiring preprocessing to transform them into probability measures. We derived a closed-form expression for the Fr\'echet derivative of the map $f \mapsto d_\text{HV}^2(f, g)$, enabling seamless integration of the HV metric into adjoint-state methods tailored for any large-scale PDE-constrained optimization tasks, including the FWI workflows. 

A spectral analysis of the Hessian operator for the HV objective functional reveals that by tuning parameters $\kappa$, $\lambda$, and $\varepsilon$, the objective functional smoothly interpolates between the $L^2$, $H^{-1}$, and $H^{-2}$ regimes. This flexibility allows the HV metric to strike a balance between local, pointwise matching and global, transport-driven alignment. Remarkably, the HV metric can be \emph{weaker} than the Wasserstein distance, whose linearization is related to the weighted $\dot{H}^{-1}$ semi-norm. Under the appropriate asymptotic conditions, the HV metric approximates the $H^{-2}$ norm, providing additional convexification when required, thereby improving its robustness and applicability in challenging inversion tasks.

Our numerical experiments confirm these theoretical insights, showing that the HV misfit mitigates cycle-skipping, accelerates convergence, and is more robust to poor initial models than inversions driven by least-squares or Wasserstein objectives. Importantly, the HV metric acts directly on signed waveforms, eliminating the preprocessing required by optimal-transport approaches and thereby preserving critical amplitude and polarity information of the waveform data.

These findings open several avenues for future work. First, automatic strategies for selecting or adapting the hyperparameters could further boost performance and eliminate the need for manual tuning. Second, a systematic study of the metric's sensitivity to measurement noise would clarify its applicability to practical inverse problem setups. Finally, the HV metric's utility may extend well beyond waveform inversion. For instance, tasks such as ECG signal registration and classification would greatly benefit from transport-based comparisons of signed, oscillatory data.

\section*{Statements and Declarations}
\subsection*{Acknowledgements} M.N.\ and Y.Y.\ acknowledge support from the National Science Foundation under grant DMS-2409855 and from the Office of Naval Research under grant N00014-24-1-2088. M.N.\ is additionally grateful to Dr.~John Guckenheimer for generously providing financial assistance during a temporary funding disruption.  
The authors thank Prof.~Dejan Slep\v{c}ev for valuable discussions.
\subsection*{Conflict of interest}The authors have no relevant financial or non-financial interests to disclose.

\bibliography{sn-bibliography}
\end{document}